\documentclass{article}

\PassOptionsToPackage{square}{natbib}
\usepackage[main, final]{neurips_2026}
\usepackage{multirow}
\usepackage[utf8]{inputenc}
\usepackage[T1]{fontenc}
\usepackage{url}
\usepackage{amsfonts}
\usepackage{nicefrac}
\usepackage{microtype}
\usepackage[dvipsnames]{xcolor}
\definecolor{green}{HTML}{4CAF50}
\definecolor{blue}{HTML}{0D47A1}

\usepackage{amsmath}
\usepackage{amssymb,amsbsy,amsthm}
\usepackage{subcaption}
\theoremstyle{plain}
\usepackage{graphicx}
\usepackage{enumerate}
\usepackage{cases}
\usepackage{thm-restate}

\usepackage{wrapfig}
\usepackage[colorlinks,linkcolor=blue,citecolor=blue,urlcolor=green]{hyperref}

\allowdisplaybreaks[1]

\numberwithin{equation}{section}

\usepackage{booktabs}
\usepackage{bm}
\usepackage{algorithm}
\usepackage{algorithmic}
\usepackage{tikz}
\tikzset{dot/.style={circle,fill=black,inner sep=0pt,minimum size=4pt}}

\usepackage{mathtools}

\usepackage{xspace}
\DeclareRobustCommand{\modelname}{\texttt{PNKG}\xspace}
\DeclareMathOperator{\R}{\mathbb{R}}

\renewcommand{\P}{\mathbb{P}}
\newcommand{\E}{\mathbb{E}}

\newcommand{\eps}{\varepsilon}

\newcommand{\op}{\mathrm{op}}

\def\cF{{\mathcal F}}

\def\cX{{\mathcal X}}

\theoremstyle{plain}
\newtheorem{theorem}{Theorem}

\newtheorem{lemma}[theorem]{Lemma}

\theoremstyle{definition}
\newtheorem{assumption}{Assumption}

\theoremstyle{remark}

\newtheorem{remark}{Remark}

\title{Toward a Unified Statistical Theory of Multi-View Pretraining for Neural Knowledge Graph Learning}
\author{%
  Jifan Zhang \\
  Department of Statistics and Data Science\\
  Northwestern University \\
  \And
  Mikl\'os Z.~R\'acz \\
  Department of Statistics and Data Science\\
  Department of Computer Science\\
  Northwestern University \\
  \And
  Suqi Liu \\
  Department of Statistics\\
  University of California, Riverside \\
}

\begin{document}
\maketitle
\begin{abstract}
Knowledge graph learning offers a powerful framework for representing and inferring structured knowledge, with broad practical applications.
Two obstacles limit current methods: labeled triples are scarce for each entity and relation, which makes expressive models hard to train, and scoring functions are designed ad hoc, which limits their generalizability and leaves them without theoretical grounding.
We address both obstacles with a theoretically grounded framework, with end-to-end guarantees, that extends and subsumes existing methods.
The framework proceeds in two stages: multi-view unsupervised pretraining that integrates heterogeneous sources of entity side information into a shared representation, followed by supervised learning across multiple relation types.
We establish a unified nonasymptotic risk bound that combines the subspace recovery error of pretraining and the sample complexity of the labeled data in a single guarantee, quantifying formally how large-scale unlabeled data from multiple views benefits downstream knowledge prediction.
Synthetic experiments validate each component of the theory, and experiments on large-scale real-world knowledge graph benchmarks confirm the effectiveness of our approach.
\end{abstract}

\section{Introduction}
Knowledge graphs (KGs) organize information as a network of entities
and the relations between them.
They have become the standard representation of relational knowledge
in domains ranging from question answering and recommendation to
information extraction and biomedical discovery
\citep{yang2022survey,zhang2016collaborative,huang2019knowledge,bonner2022review}.
Large-scale KGs such as the Google Knowledge Graph and
the Unified Medical Language System (UMLS) contain millions to billions of
entities spanning dozens of heterogeneous relation types, making scalable and
accurate KG prediction a central problem in machine learning \citep{ji2021survey}.
Despite their scale, real-world KGs are sparse: the observed triples cover only
a small fraction of all possible entity pairs, and this sparsity hinders the
learning of KG prediction models.
At the same time, entities in these graphs usually carry rich unlabeled side
information—textual descriptions, ontology memberships, database annotations—that
is readily available even when labeled triples are scarce.

The dominant approach to KG prediction is embedding-based: learn low-dimensional
entity and relation vectors so that a translational, bilinear, or rotational
scoring function ranks observed triples above negatives
\citep{bordes2013translating,yang2014embedding,trouillon2016complex,sun2019rotate}.
A substantial empirical literature further augments these representations with
textual descriptions, pretrained language models, or other entity attributes
\citep{xie2016representation,yao2019kg,wang2021kepler,wang2022simkgc}, demonstrating
that side information is valuable in practice, especially when labeled triples are scarce.
However, these methods are typically tailored to specific applications, rely on
domain-specific inductive biases, and largely lack theoretical guarantees.
Recently, \citet{liu2024representation} established nonasymptotic risk bounds for
neural KG learning with fixed and trainable entity representations, and
\citet{ge2023provable} showed that unsupervised pretraining reduces labeled-sample
complexity in generic latent-variable models.
We build on these developments to establish a unified statistical
framework connecting unsupervised pretraining to supervised KG prediction.

A fundamental challenge in KG learning is the tension between data scarcity and
model capacity: labeled triples are typically limited, yet flexible
representation-based models require many parameters.
How to systematically exploit diverse unlabeled datasets with heterogeneous
side information, and what provable benefits doing so confers, therefore
remain critical open questions.
We address these questions with a theoretically grounded two-stage framework,
which we call the Pretrained Neural Knowledge Graph (PNKG) framework.
In the first stage, side information on entities from multiple views is
mapped through view-specific positive-semidefinite kernels, and kernel principal
component analysis (KPCA) on a weighted aggregate of their Gram matrices
yields low-dimensional entity coordinates.
The kernel formulation also offers a natural interface to pretrained
language encoders: each frozen encoder induces a positive-semidefinite kernel
on entity descriptions through the inner product of its feature map, consistent
with the kernel perspective on deep architectures of \citet{cho2009kernel}.
The PNKG framework aggregates these encoder-induced Gram matrices, so that views
with different architectures and output dimensions can be combined on their
common entity index set.
In the second stage,
these embeddings serve as inputs to ReLU networks trained on labeled triples.
The weighted direct sum of the view-specific reproducing kernel Hilbert spaces
(RKHSs) provides a unified representation space for heterogeneous side
information, such as text descriptions and metadata, drawn from diverse datasets,
while the ReLU networks model the scoring function flexibly without
predefined domain-specific priors.

We establish two main theoretical results for this framework.
Theorem~\ref{thm:main-pretrain} gives a nonasymptotic perturbation bound for
the KPCA pretraining stage, showing that the learned entity subspace is accurate
when the multi-view kernel residuals are small relative to the effective signal
separation of the weighted population geometry. The result covers both
inner-product and Gaussian distance kernels in admissible
information-plus-noise regimes, and it separates finite-sample kernel error from
the population bias that a misspecified kernel--noise geometry may induce.
It also motivates a signal-adjusted inverse-variance principle for weighting views:
views with strong effective signal, small deterministic bias, and low
fluctuation variance should receive larger weights, and noisy or
systematically biased views should be downweighted.
This subspace recovery result supplies the pretraining component of our
end-to-end statistical theory and may be of independent interest.
Our main result,
Theorem~\ref{thm:main-end2end}, is a unified end-to-end risk bound that
decomposes the downstream KG prediction error additively into neural
approximation error, supervised estimation error, optimization error,
and pretraining error.
It makes explicit which term is the bottleneck as the unlabeled and labeled
sample sizes and the model capacity vary, and it attributes each source of
error to a specific stage of the pipeline or a specific modeling or
algorithmic choice.

Simulation studies on synthetic data further support the theoretical findings.
We also evaluate the framework on two real-world KGs:
WordNet~\citep{miller1995wordnet}, a lexical KG whose synset definitions
provide textual side information, and RENKI-KG~\citep{liu2024representation},
a large-scale medical KG that integrates multiple curated clinical
ontologies and drug databases.
On both, multi-view pretrained neural KG models outperform graph-only
baselines and improve on their single-view counterparts.\footnote{Code is available at \url{https://github.com/jifanzhang999/PNKG}.}
These results corroborate the theory and demonstrate the practical utility of the framework.

In summary, our contributions are threefold.
(i) We propose a generic two-stage framework that combines multi-view KPCA for
unsupervised pretraining with ReLU networks for supervised KG learning.
(ii) We prove an end-to-end transfer theorem with nonasymptotic bounds that
connects pretraining error to supervised KG prediction risk.
(iii) We validate the theory with controlled simulations that isolate each
error component and with real-world experiments on benchmark KGs.

\section{Related Work}
\textbf{Knowledge graph embeddings.}
Most KG completion methods learn entity and relation embeddings together with a relation-specific scoring function. TransE \citep{bordes2013translating} models relations as translations; DistMult \citep{yang2014embedding} and ComplEx \citep{trouillon2016complex} use bilinear interactions; and RotatE \citep{sun2019rotate} represents relations as rotations in complex space. \citet{nickel2016review} and \citet{ji2021survey} survey these approaches. Although they are effective graph-only baselines, they rely mainly on observed triples and do not directly exploit multiple unlabeled sources of entity side information, which can be valuable when labels for the target relations are scarce or entity coverage is uneven.

\textbf{Text- and side-information-enhanced KG completion.}
Several approaches enrich KG prediction with entity descriptions or pretrained language models. DKRL \citep{xie2016representation} encodes entity descriptions alongside KG embeddings, KG-BERT \citep{yao2019kg} scores triples written as text, and KEPLER \citep{wang2021kepler} learns textual and KG representations jointly. BLP \citep{daza2021inductive} learns inductive representations, StAR \citep{wang2021star} couples textual encoding with structure-aware scoring, and SimKGC \citep{wang2022simkgc} applies contrastive learning to text-based KG completion. Rather than proposing another scoring architecture, we study how representations estimated from \emph{unlabeled, heterogeneous views} affect the risk of a downstream relation-specific predictor. Our guarantees cover this two-stage setting with fixed pretrained coordinates, not end-to-end fine-tuning of the encoders; this separation isolates the statistical contribution of unsupervised representation estimation.

\textbf{Statistical theory for KG and network learning.}
Latent-position network models \citep{hoff2002latent} and shared-structure models for multilayer networks \citep{macdonald2022latent} provide theoretical foundations for estimating latent geometry, typically under specified link functions or structured observation models. Such assumptions can be restrictive for KGs with heterogeneous relations and unknown scoring functions. Most closely related, \citet{liu2024representation} formulate neural KG prediction as nonparametric regression over triples and derive nonasymptotic in-sample and out-of-sample MSE bounds. We complement that supervised analysis: instead of treating the entity coordinates as observed features, we quantify the error in recovering them from unlabeled multi-view data and how this error propagates to the downstream KG risk.

\textbf{Theory of unsupervised pretraining.}
Theories of contrastive representation learning \citep{arora2019theoretical,haochen2021provable} establish downstream benefits under latent-class or augmentation assumptions, and prediction-based pretraining \citep{lee2021predicting} can reduce the number of labeled samples needed under suitable structural conditions. Closest in spirit, \citet{ge2023provable} separate the unlabeled and labeled contributions to excess risk under an informative-representation condition. Whereas such generic analyses abstract away the representation estimator, we analyze a specific one, weighted kernel PCA over multiple views, in a KG setting with heterogeneous views and relation-labeled triples: Theorem~\ref{thm:main-end2end} tracks the multi-view pretraining error through a neural KG regression bound. Because the models and assumptions differ, rates are not directly comparable across these analyses.

\textbf{Deep feature kernels and spectral aggregation.}
A frozen encoder defines a feature map \(\phi_s\) and an induced kernel \(k_s(u,u')=\langle\phi_s(u),\phi_s(u')\rangle\), so inner products of frozen language-model embeddings evaluate an encoder-induced kernel. This links our procedure to the deep-kernel perspectives of \citet{cho2009kernel} and \citet{wilson2016deep}. Our aggregation also shares with distributed PCA \citep{fan2019distributed} the goal of recovering a common spectral subspace, but the two operate differently: distributed PCA averages local eigenspace projectors, whereas PNKG aggregates the view-specific entity-kernel matrices \emph{before} the eigendecomposition. This distinction motivates our analysis of the view weights and the signal eigengap.

\section{The PNKG Framework}
\label{sec:setup}
We now formalize the two-stage PNKG framework.  The first stage uses unlabeled
side information on the entities from multiple sources to construct
low-dimensional entity coordinates.  The second stage takes these
coordinates as inputs and learns relation-specific ReLU networks from
labeled triples.
\subsection{Multi-view pretraining via weighted KPCA}
\label{subsec:pretraining-model}

Let \([N]:=\{1,\ldots,N\}\) denote the common entity set.
Side information is collected from several sources, called views, and
\([S]:=\{1,\ldots,S\}\) indexes them.
For each view $s\in[S]$,
the side information is organized into a matrix
\[
    E_s =
    \begin{pmatrix}
     u_{s,1}^{\top} \\
    \vdots\\
     u_{s,N}^{\top}
    \end{pmatrix}
    \in \R^{N\times m_s},
\]
where $m_s$ is the feature dimension of view $s$ and $u_{s,j}\in\R^{m_s}$
is the data vector of entity $j$ in view $s$.
For example, view $s$ may be a hospital with $m_s$ patients and
$j$ may index diagnoses, so that $u_{s,j}$ records information on all patients
with diagnosis $j$ in the medical records of hospital $s$.

Our theoretical analysis is motivated by an information-plus-noise
structure: for each view $s$, we model the observed entity vector as
\begin{equation}
    u_{s,j}
    =
    v_{s,j}
    +
    \sigma_{s,j}\frac{\varepsilon_{s,j}}{\sqrt{m_s}},
    \qquad
    j=1,\ldots,N,
    \label{eq:info-plus-noise-main}
\end{equation}
where $v_{s,j}\in\R^{m_s}$ is the latent signal vector of entity $j$,
$\varepsilon_{s,j}\in\R^{m_s}$ is high-dimensional noise, and $\sigma_{s,j}$ is a
view-specific noise scale.  The normalization $m_s^{-1/2}$ keeps the noise
energy bounded as the feature dimension grows.

\textbf{A linear kernel example.}
We first specialize \eqref{eq:info-plus-noise-main} to the simplest,
linear isotropic setting. Assume that the noise scale is constant within each
view, $\sigma_{s,j} = \sigma_s$, and that
\[
    \varepsilon_{s,j}\sim N(0,I_{m_s}),\qquad j=1,\ldots,N.
\]
Then
\[
    E_s
    =
    V_s
    +
    \frac{\sigma_s}{\sqrt{m_s}}\xi_s,
    \qquad
    \xi_s(j,\cdot)\stackrel{\rm iid}{\sim}N(0,I_{m_s}),
\]
where the $j$th row of $V_s$ is $v_{s,j}^{\top}$. The corresponding linear
entity-similarity matrix is
\begin{equation}
\label{eq:linear}
    K_s^{\rm lin}
    =
    \frac{1}{N}E_sE_s^\top
    \in\R^{N\times N}.
\end{equation}
With multiple views, we aggregate the entity-similarity matrices as
\[
     K^{\rm lin}(w)
    =
    \sum_{s=1}^S w_s K_s^{\rm lin},
    \qquad
    w_s\ge0,\quad \sum_{s=1}^S w_s=1.
\]
The leading eigenvectors of $K^{\rm lin}(w)$ then provide
low-dimensional entity coordinates.

\textbf{General kernels.} The linear construction extends to any
positive-semidefinite kernel. For each view \(s\in[S]\), let
\[
    k_s:\mathbb R^{m_s}\times\mathbb R^{m_s}\to\mathbb R
\]
be a positive-semidefinite kernel. There then exist a Hilbert space
\(\mathcal H_s\) and a feature map \[
    \phi_s:\mathbb R^{m_s}\to\mathcal H_s
\]
such that
\[
    k_s(u,u')
    =
    \left\langle
        \phi_s(u),\phi_s(u')
    \right\rangle_{\mathcal H_s}.
\]
We define the view-specific entity-kernel matrix by
\[
    K_s^{\rm obs}(j,j')
    :=
    \frac1N
    k_s(u_{s,j},u_{s,j'})
    =
    \frac1N
    \left\langle
        \phi_s(u_{s,j}),
        \phi_s(u_{s,j'})
    \right\rangle_{\mathcal H_s},
    \qquad
    j,j'\in[N].
\]
Thus \(K_s^{\rm obs}\) captures the (possibly nonlinear) geometry that view \(s\) induces on the
common entity set.
Examples include inner-product
kernels of the form
\(
    k_s(u,u')
    =
    f_s(u^\top u')
\)
and distance kernels of the form
\(
    k_s(u,u')
    =
    f_s(\|u-u'\|_2^2)
\). The linear construction above is the special case
\(
    k_s(u,u')=u^\top u',
    \,
    \phi_s(u)=u
\).

For weights $w=(w_1,\ldots,w_S)$ in the probability simplex $\Delta_S$, define the weighted
multi-view entity-kernel matrix
\begin{equation}
    K^{\rm obs}(w)
    :=
    \sum_{s=1}^S w_s K_s^{\rm obs},
    \qquad
    w_s\ge0,\quad \sum_{s=1}^S w_s=1 .
    \label{eq:weighted-entity-kernel}
\end{equation}
Because the aggregation is performed on the common entity index set, it is
well defined even when the input spaces and feature dimensions of the views
differ, and it is itself a positive-semidefinite kernel on the common entity set.
Appendix~\ref{app:direct-sum} gives the weighted direct-sum construction
and interprets frozen encoders as feature maps
(Remark~\ref{rem:encoder-induced-kernels}); Appendix~\ref{app:linear-example}
revisits the linear example.

\textbf{Spectral entity coordinates.} The pretrained entity coordinates come from the leading eigenspace of
$K^{\rm obs}(w)$. For a symmetric matrix \(G\), let
\({\rm TopEig}_d(G)\) denote any \(N\times d\) matrix whose columns form an
orthonormal basis of the eigenspace associated with the \(d\) largest eigenvalues
of \(G\).  Let
\[
    \widehat Z
    :=
    {\rm TopEig}_d\!\left(K^{\rm obs}(w)\right)
    \in\R^{N\times d},
    \qquad
    \widehat Z^\top \widehat Z=I_d .
\]
The supervised stage uses \(\widehat Z\) as the frozen entity-coordinate matrix,
whose \(j\)th row is the pretrained coordinate of entity \(j\); the
choice of \(d\) is discussed in Remark~\ref{rem:choice-d}.
Our analysis assumes that the population signal kernels share a rank-$d$ eigenspace
$Z_{\rm sig}$, with view-specific signal strengths $\Lambda_s$ and residuals
$R_s^{\rm sig}$ on its orthogonal complement; see \eqref{eq:shared-subspace-main}
in Appendix~\ref{app:shared-signal}.
Algorithm~\ref{alg:multi-view-kpca} summarizes the multi-view kernel pretraining procedure.
\begin{algorithm}[t]
\caption{Multi-view kernel pretraining}
\label{alg:multi-view-kpca}
\begin{algorithmic}[1]
\REQUIRE Entity-feature observations $\{u_{s,j}\in\R^{m_s},j\in[N]\}_{s=1}^S$,
kernels $\{k_s\}_{s=1}^S$, rank $d$, and weights $w\in\Delta_S$.
\ENSURE Frozen entity coordinates $\widehat Z\in\R^{N\times d}$.
\FOR{$s=1,\ldots,S$}
    \STATE $ K_s^{\rm obs}(j,j')\leftarrow \frac 1 N k_s(u_{s,j},u_{s,j'})$ for all $j,j'\in[N]$.
\ENDFOR
\STATE $K^{\rm obs}(w)\leftarrow \sum_{s=1}^S w_s K_s^{\rm obs}$.
\STATE $\widehat Z\leftarrow{\rm TopEig}_d\!\left( K^{\rm obs}(w)\right)$.
\STATE \textbf{return} $\widehat Z$.
\end{algorithmic}
\end{algorithm}

\subsection{Supervised learning with relation-wise networks}
\label{subsec:supervised-model}

The supervised stage uses labeled (head, relation, tail) triples,
abbreviated as $(h, r, t)$ where $h$ and $t$ are entities and $r$ is the relation between them.
Let $[K]=\{1,\ldots,K\}$ be the
relation set and
\(
    \mathcal X=[N]\times[K]\times[N]
\)
be the space of all possible triples.  We observe labeled samples
$\mathcal D_{\rm sup}=\{(X_i,Y_i)\}_{i=1}^n$, where
\begin{equation}
X_i=(h_i,r_i,t_i)\in\mathcal X
\quad\text{and}\quad
Y_i=f^\star(X_i)+\varepsilon_i.
\label{eq:supervised-label-model}
\end{equation}
We assume that $\E[\varepsilon_i\mid X_i]=0$.
The true target function $f^\star:\mathcal X\to\R$ is the KG scoring function.
For binary link prediction, $f^\star(X_i)$ can be interpreted as the conditional probability
that triple $X_i$ is true.

Given the pretrained embeddings
$\widehat Z=(\widehat z_1,\ldots,\widehat z_N)^\top$, we learn one bounded ReLU network for each relation.
Let $\Pi_B:\R\to[-B,B]$ be the clipping map. For relation $r$, let
$F_{r}:\R^{2d}\to\R$  be a ReLU network with depth at most $D$ and at most $W$ trainable parameters, and define
\begin{equation}
f_{r}(a,b)
\coloneqq
\Pi_B\{F_{r}([a^\top,b^\top]^\top)\}.
\label{eq:relation-wise-head}
\end{equation}
The neural knowledge graph model is
\begin{equation}
f(\widehat Z;h,r,t)
= f_{r}(\widehat z_h,\widehat z_t).
\label{eq:full-frozen-predictor}
\end{equation}
Let $\mathcal F_{\widehat Z}(D,W,B)$ denote the neural knowledge graph function class given $\widehat Z$.
We work with a
predictor \(\widehat f(\widehat Z;\cdot)\in
\mathcal F_{\widehat Z}(D,W,B)\) whose empirical risk is within
\(\delta_{\rm opt}\) of the best empirical risk over the class
\begin{equation}
\frac1n\sum_{i=1}^n
\{Y_i-\widehat{f}(\widehat Z;X_i)\}^2
\le
\inf_{f\in\mathcal F_{\widehat Z}(D,W,B) }
\frac1n\sum_{i=1}^n
\{Y_i-f(\widehat Z;X_i)\}^2
+\delta_{\rm opt}.
\label{eq:supervised-approx-erm}
\end{equation}
In practice, $\widehat f$ is computed with standard optimizers such as stochastic gradient descent.

\section{Theoretical Results}
\label{sec:theory}
Building on the two-stage construction in Section~\ref{sec:setup}, we first
bound the subspace error of the weighted KPCA coordinates $\widehat Z$ and
then propagate this error to the downstream relation-specific predictor.
We introduce only the quantities needed to state the guarantees;
the precise kernel-dependent definitions and the proofs are in
Appendices~\ref{app:theory-details} and~\ref{app:proofs}.

\subsection{Subspace recovery for multi-view pretraining}
\label{subsec:pretrain-theory}

\textbf{Assumptions.}
We work under the information-plus-noise model~\eqref{eq:info-plus-noise-main}
with the weighted kernel~\eqref{eq:weighted-entity-kernel}. The central
structural condition is a \emph{shared signal eigenspace}: the noiseless
kernel matrices $K_s^{\rm sig}(j,j'):=N^{-1}k_s(v_{s,j},v_{s,j'})$ decompose as
\[
K_s^{\rm sig}=Z_{\rm sig}\Lambda_sZ_{\rm sig}^\top+R_s^{\rm sig},
\qquad s\in[S],
\]
where $Z_{\rm sig}\in\R^{N\times d}$ with $Z_{\rm sig}^\top Z_{\rm sig}=I_d$
spans the common $d$-dimensional entity eigenspace, the positive-semidefinite
matrix $\Lambda_s\in\R^{d\times d}$ gives the signal strength of view $s$
along it, and the symmetric residual $R_s^{\rm sig}$, which satisfies
$R_s^{\rm sig}Z_{\rm sig}=0$, captures view-specific variation in its
orthogonal complement. The views need not be equally informative;
they need only share this leading eigenspace.
In addition, the recovery guarantee requires the noise in each view to
satisfy suitable concentration or Gaussian conditions, the shared signal to
be separated from the residual directions, and the views to be independent.
The precise statements are
Assumptions~\ref{ass:ek-info-noise}--\ref{ass:pretrain-supervised-independence}
in Appendix~\ref{app:theory-details}. Throughout this subsection, the weights
$w\in\Delta_S$ are fixed independently of the pretraining data.

\textbf{Kernel-dependent errors and effective eigengap.}
Noise affects each kernel family differently. Let
$\mathcal S_{\rm lin}$, $\mathcal S_{\rm gauss}$, and $\mathcal S_{\rm ip}$
denote the sets of linear, Gaussian distance-kernel, and nonlinear
inner-product views, respectively. Four view-dependent quantities enter
the bound:
\begin{itemize}
    \item $\alpha_s>0$ is a \emph{signal multiplier} that captures the attenuation
    caused by the kernel and the noise. It is $1$ for linear and nonlinear
    inner-product kernels, and $e^{-2\tau_s\sigma_s^2}$ for a Gaussian
    kernel with inverse bandwidth $\tau_s$ and homogeneous noise scale $\sigma_s$.
    \item $b_{s,N,m_s}\ge0$ measures the \emph{deterministic kernel bias} of
    the population approximation to the noisy kernel.
    \item $d_{s,N,m_s}\ge0$ measures a \emph{nonconstant diagonal
    correction} for nonlinear inner-product kernels; it vanishes in the
    other two regimes.
    \item $v_{s,N,m_s}\ge0$ measures the \emph{random kernel fluctuation}
    around its population approximation; $v_{s,N,m_s}^2$ is the
    variance scale that enters the bound.
\end{itemize}
The kernel-specific formulas, including the required noise parameters,
are given in Appendix~\ref{app:kernel-regimes}.

The separation between the \emph{weighted} signal and residual parts
is measured by
\begin{equation}
\gamma(w)
:=\lambda_{\min}\!\left(\sum_{s=1}^S w_s\alpha_s\Lambda_s\right)
-\left\|\sum_{s=1}^S w_s\alpha_s R_s^{\rm sig}\right\|_{\op}.
\label{eq:theory-effective-gap}
\end{equation}
Here $\lambda_{\min}$ is the smallest
eigenvalue of the $d\times d$ weighted signal matrix and $\|\cdot\|_{\op}$ is the spectral norm.
Thus $\gamma(w)$ is a \emph{lower bound on the effective eigengap}:
the larger it is, the more clearly the common signal subspace stands out
from the residual directions.
The separation assumption guarantees $\gamma(w)>0$ for the weights
under consideration.

We collect the two types of error into
\begin{align}
\mathcal B(w)
&:=\sum_{s=1}^S w_s b_{s,N,m_s}
 +\sum_{s\in\mathcal S_{\rm ip}}w_s d_{s,N,m_s},\qquad
\mathcal V(w)
:=\sum_{s=1}^S w_s^2v_{s,N,m_s}^2,
\label{eq:theory-variance-aggregate}
\end{align}
and define the normalized pretraining-error quantity
\begin{equation}
\mathcal E_{\rm pt}(w,\delta)
:=\frac{\mathcal B(w)^2+\delta^{-1}\mathcal V(w)}{\gamma(w)^2},
\qquad 0<\delta<1.
\label{eq:theory-pretrain-error}
\end{equation}
Here $\delta$ controls the failure probability. The squared
weights in $\mathcal V(w)$ reflect the independence of the fluctuations
across views, whereas the deterministic errors accumulate linearly
before being squared.

\begin{theorem}[Subspace recovery for weighted entity-kernel PCA]
\label{thm:main-pretrain}
Suppose Assumptions~\ref{ass:ek-info-noise}--\ref{ass:pretrain-supervised-independence}
hold and $w\in\Delta_S$ is deterministic. For every $\delta\in(0,1)$,
with probability at least $1-\delta$,
\begin{equation}
\big\|\sin\Theta(\widehat Z,Z_{\rm sig})\big\|_F^2
\le C\,\mathcal E_{\rm pt}(w,\delta)
=\frac{C}{\gamma(w)^2}
\left[\mathcal B(w)^2+\frac{\mathcal V(w)}{\delta}\right].
\label{eq:main-pretrain-bound}
\end{equation}
Here $C>0$ depends only on the view-level noise constants, not on \(N\), \(m_s\), \(w\), or \(\delta\).
\end{theorem}

The matrix $\sin\Theta(\widehat Z,Z_{\rm sig})$ has as its singular
values the sines of the principal angles between the estimated and target
$d$-dimensional subspaces. Theorem~\ref{thm:main-pretrain} separates
(i) the deterministic bias $\mathcal B(w)^2$,
(ii) the random fluctuation $\mathcal V(w)/\delta$, and
(iii) the amplification of both by the inverse squared eigengap
$\gamma(w)^{-2}$. A view is useful not simply because it
has many features or low noise, but because its noise is small relative
to the signal separation it contributes.

\textbf{Rates and view weighting.}
Let $M=\sum_{s=1}^S m_s$ be the total feature dimension across views.
Under the bounded-parameter conditions specified in
Appendix~\ref{app:weighting}, the fluctuation scales typically satisfy
$v_{s,N,m_s}^2=O(m_s^{-1})$ and $b_{s,N,m_s}=O(m_s^{-1})$.
With the simple sample-size weights $w_s=m_s/M$, at fixed confidence
and for an effective gap bounded away from zero, the squared subspace
error is $O(M^{-1})$ when no nonconstant diagonal correction is present.
Nonlinear inner-product kernels may contribute an additional
$O(N^{-1})$ term.
More generally, let
\begin{equation}
\gamma_s:=\lambda_{\min}(\Lambda_s)-\|R_s^{\rm sig}\|_{\op},
\qquad g_s:=\alpha_s\gamma_s,
\label{eq:theory-per-view-gap}
\end{equation}
where $\gamma_s$ (assumed positive) is the signal--residual separation of \emph{view $s$ alone}
and $g_s$ also accounts for the noise-induced
signal multiplier. Writing $v_s=v_{s,N,m_s}$ for brevity, an idealized
signal-adjusted inverse-variance rule sets
\begin{equation}
w_s^{\rm IV}\propto\frac{g_s}{v_s^2}.
\label{eq:theory-iv-weights}
\end{equation}
This rule optimizes a variance-only surrogate of the error bound,
ignoring the lower-order bias term. The quantities $\Lambda_s$, $R_s^{\rm sig}$, and the
fluctuation scales are unknown population
quantities, but they can be estimated from the observed data
under suitable assumptions. Thus
\eqref{eq:theory-iv-weights} is an oracle
principle that also motivates practical
weighting estimators, as our simulations illustrate; sample-size weighting, in contrast, uses only the observable $m_s$.
Appendix~\ref{app:weighting} gives the surrogate optimization, the rates,
and kernel-specific qualifications. Both rules favor
reliable, well-separated views.

\subsection{End-to-end risk bound}
\label{subsec:end-to-end-theory}
\label{subsec:supervised-assumptions}

\textbf{Assumptions.}
We retain the supervised triple model~\eqref{eq:supervised-label-model} and
the clipped relation-wise ReLU heads from
Section~\ref{subsec:supervised-model}. To obtain a population risk bound, we
impose standard conditions on the sampling mechanism and the target function.
The labeled triples are i.i.d.\ from a distribution $P$ whose head and tail
marginals are uniform on $[N]$, and the label noise
is sub-Gaussian with variance proxy $\sigma_y^2$. The true scores are bounded
by $B$ and are smooth functions of latent entity coordinates
$Z^\star\in\R^{N\times d}$. The key condition linking the two stages is an
\emph{alignment condition}: there exists an invertible $d\times d$ matrix
$A_\star$ such that
\[
    \sqrt N\,Z_{\rm sig}=Z^\star A_\star,
    \qquad
    \max\bigl\{\|A_\star\|_{\op},\|A_\star^{-1}\|_{\op}\bigr\}\le\kappa_A,
\]
where the factor $\sqrt N$ keeps the latent coordinates of order one.
This transfer assumption requires the latent subspace relevant for KG
prediction to agree with the population subspace recovered by pretraining, up
to a well-conditioned change of coordinates; reducing the pretraining noise
alone does not imply it. The precise statements are
Assumptions~\ref{ass:sup-sampling-app}--\ref{ass:alignment-app} in
Appendix~\ref{app:supervised-assumptions}.

\textbf{Risk and complexity.}
We measure the prediction error of a predictor $q:\cX\to\R$ on a new triple by
its population risk
\[
\|q-f^\star\|_P^2
:=\E_{X\sim P}\{q(X)-f^\star(X)\}^2 .
\]
The statistical complexity of the frozen-embedding predictor class is
measured by its \emph{pseudo-dimension}; for the relation-wise networks of
Section~\ref{subsec:supervised-model}, it is at most
$p:=C_{\rm pd}KDW\log W$ for an absolute constant $C_{\rm pd}\ge1$
(Lemma~\ref{lem:fixed-oracle-app}). Let $L_\star$ be a common Lipschitz constant of the
true scores as functions of the latent coordinates; it is controlled by the
smoothness norm in Assumption~\ref{ass:smooth-kg-app}. Finally, the
pretraining error enters through $\mathcal E_{\rm pt}(w,\delta)$ defined
in~\eqref{eq:theory-pretrain-error}.

\begin{samepage}
\begin{theorem}[End-to-end oracle inequality for pretrained neural KG learning]
\label{thm:main-end2end}
Suppose Assumptions~\ref{ass:ek-info-noise}--\ref{ass:alignment-app}
hold, $w\in\Delta_S$ is deterministic, $p=C_{\rm pd}KDW\log W\le n$, $W\ge W_0$, and $D\gtrsim\log W$. With probability at least
$1-10(p/en)^p-\delta$ over the unlabeled pretraining and labeled
supervised samples,
\begin{align}
\|\widehat f(\widehat Z;\cdot)-f^\star\|_P^2
\le{}&
\underbrace{C_1\left(\frac{W}{\log W}\right)^{-\beta/d}}
_{\text{neural approximation error}}
\notag\\
&+\underbrace{C_2(\sigma_y^2+B^2)
\frac{KDW\log W}{n}
\log\frac{en}{KDW\log W}}
_{\text{supervised estimation error}}
\notag\\
&+\underbrace{C_3(L_\star^2+B^2)\kappa_A^2\mathcal E_{\rm pt}(w,\delta)}
_{\text{pretraining error}}
+\underbrace{C_4\delta_{\rm opt}}_{\text{optimization error}}.
\label{eq:main-end2end-bound}
\end{align}
The constants $C_1,\ldots,C_4$ and $W_0$ depend only on the fixed quantities
specified in Appendix~\ref{app:proof-end2end}.
\end{theorem}
\end{samepage}

\textbf{Interpretation of the four errors.}
The \emph{approximation term} decreases as each relation head becomes
more expressive (larger $W$) and as the true scoring functions become
smoother (larger $\beta$). Its exponent involves the latent dimension
$d$, not the dimensions $m_s$ of the raw side information.
The \emph{supervised estimation term} is the cost of learning $K$
relation-specific networks from only $n$ labeled triples; it grows
with network size, depth, and label noise $\sigma_y^2$.
The \emph{pretraining term} carries the subspace recovery error from
Theorem~\ref{thm:main-pretrain} into prediction risk, scaled by
$L_\star^2$ (sensitivity of the true score to latent-coordinate changes),
$B^2$ (the score range, covering poorly recovered entities), and
$\kappa_A^2$ (the alignment cost). The \emph{optimization term}
vanishes when $\delta_{\rm opt}=0$.

In particular, choosing $D\asymp\log W$ makes the
supervised estimation rate of order $KW/n$ up to logarithms, while the
approximation rate is $(W/\log W)^{-\beta/d}$. Balancing the two
gives $W\asymp(n/K)^{d/(\beta+d)}$ and the risk bound
$(n/K)^{-\beta/(\beta+d)}+(L_\star^2+B^2)\kappa_A^2\mathcal E_{\rm pt}(w,\delta)+\delta_{\rm opt}$,
up to logarithmic factors. The first term is the minimax rate for a
$\beta$-smooth function of the $2d$ latent coordinates, independent of the raw dimensions $m_s$.
This requires a small pretraining error: more labeled data cannot remove a persistent
representation mismatch, such as a failure of
Assumption~\ref{ass:alignment-app}.

\section{Experiments}
Our empirical evaluation has two parts. Section~\ref{sec:simulations} uses
controlled simulations with known ground truth to directly check the rates
and qualitative predictions of Theorems~\ref{thm:main-pretrain}
and~\ref{thm:main-end2end}: subspace recovery, view weighting, and the
end-to-end risk decomposition. Section~\ref{sec:realdata} presents the main real-data results on
WordNet and RENKI-KG, testing whether multi-view pretraining improves link prediction
under realistic textual side information and limited labeled triples.

\subsection{Simulation study}
\label{sec:simulations}
We use three controlled simulations to check these predictions.
Experiments~1 and 2 test the pretraining component of the theory:
Experiment~1 verifies the predicted subspace recovery rate,
while Experiment~2 evaluates inverse-variance view weighting under heterogeneous
view quality. Experiment~3 tests the end-to-end risk decomposition by
varying the labeled sample size \(n\), the pretraining sample size \(m\), and
the number of parameters \(W\), thereby isolating the supervised estimation,
pretraining, and approximation components of the bound.

Figure~\ref{fig:sim-pretrain-main-text} summarizes the pretraining behavior.
In Experiment~1, the subspace error $\|\sin\Theta(\widehat Z,Z_{\rm sig})\|_{\op}$ follows the predicted \(m^{-1/2}\) rate; the fitted log--log slope is $-0.506$. In
Experiment~2, the theory-guided inverse-variance weighting rule shifts mass
from the noisier views to the cleaner ones, consistent with its near-oracle
performance.
\begin{figure}[ht]
    \centering
    \begin{subfigure}[t]{0.335\textwidth}
        \centering
        \includegraphics[height=3.3cm,trim=23.76bp -36.9bp 0 -3.56bp,clip]{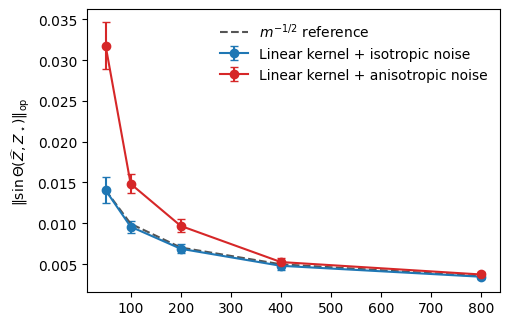}
        \caption{Subspace error vs.\ $m$ (Exp.\ 1).}
    \end{subfigure}
    \hfill
    \begin{subfigure}[t]{0.32\textwidth}
        \centering
        \includegraphics[height=3.3cm,trim=21.6bp 0 0 0,clip]{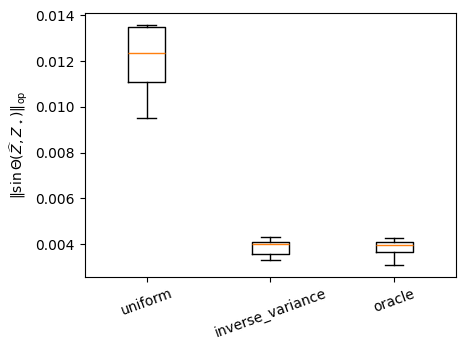}
        \caption{Error by weighting rule (Exp.\ 2).}
    \end{subfigure}
    \hfill
    \begin{subfigure}[t]{0.32\textwidth}
        \centering
        \includegraphics[height=3.3cm,trim=0 -32.46bp 0 -2.94bp,clip]{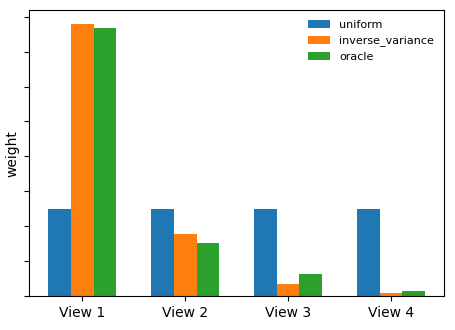}
        \caption{View weights by rule (Exp.\ 2).}
    \end{subfigure}
    \caption{Pretraining simulations.}
    \label{fig:sim-pretrain-main-text}
\end{figure}

\textbf{End-to-end supervised risk.}
Experiment~3 uses pretrained embeddings and learns bounded
relation-wise ReLU networks on a nonlinear triple-score teacher with
relation-specific bilinear interactions and Gaussian label noise
\(N(0,0.3^2)\). We vary the labeled sample size \(n\), pretraining sample size \(m\), and number of parameters \(W\); the results are shown in
Figure~\ref{fig:exp3}. The oracle baseline uses the true
latent entity coordinates from the simulator while keeping the same supervised training procedure.
As the number of labeled triples increases,
test MSE decreases consistently with the supervised-estimation term
(Figure~\ref{fig:sim-main-exp3-n-risk}).
Multi-view KPCA remains at or above the oracle, with the gap representing the downstream
cost of imperfect coordinate recovery; this gap shrinks as the pretraining sample grows
(Figure~\ref{fig:sim-main-exp3-m-risk}).
The parameter sweep reveals a bias--variance trade-off in the relation heads: larger networks improve
expressiveness but raise estimation variance when labeled data are limited
(Figure~\ref{fig:sim-main-exp3-w-risk}).
All three trends match the risk decomposition in Theorem~\ref{thm:main-end2end}, where
supervised sample complexity, approximation capacity, and pretraining error jointly determine
the final risk.

\begin{figure}[ht]
\centering
\begin{subfigure}[t]{0.32\textwidth}
\centering
\includegraphics[width=\linewidth]{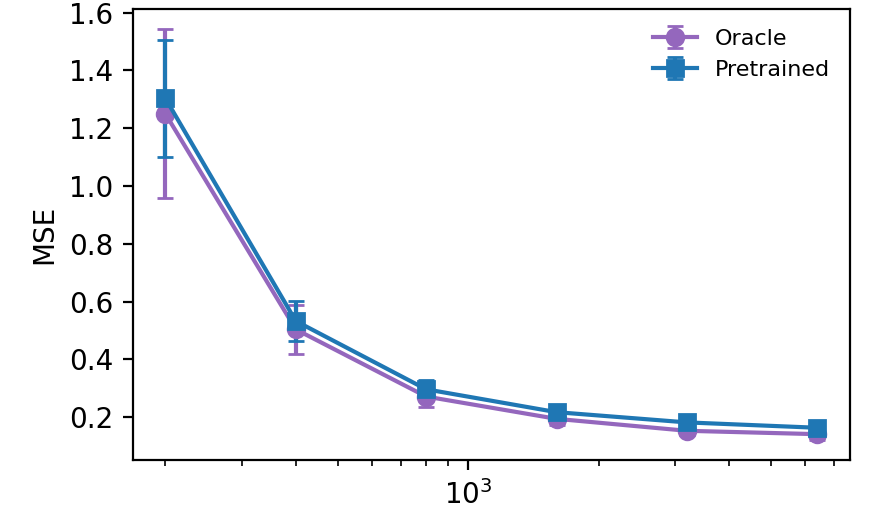}
\caption{MSE vs.\ $n$.}
\label{fig:sim-main-exp3-n-risk}
\end{subfigure}
\hfill
\begin{subfigure}[t]{0.32\textwidth}
\centering
\includegraphics[width=\linewidth]{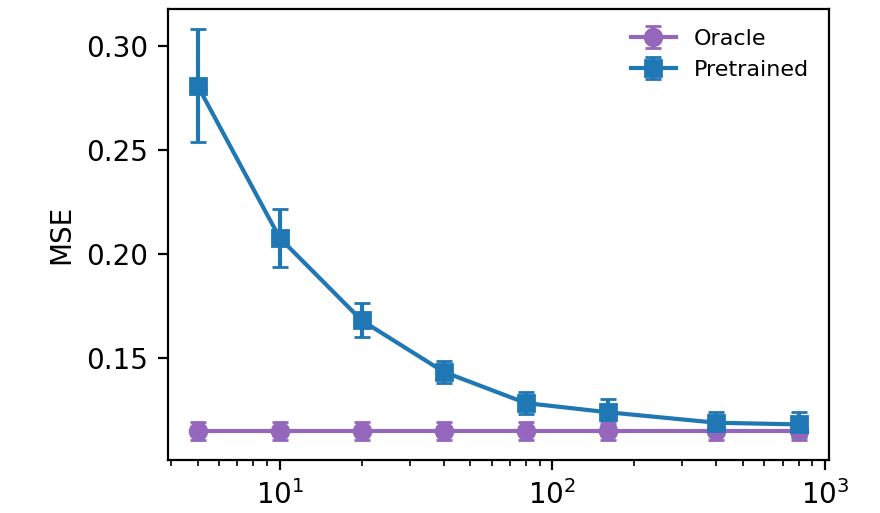}
\caption{MSE vs.\ $m$.}
\label{fig:sim-main-exp3-m-risk}
\end{subfigure}
\hfill
\begin{subfigure}[t]{0.32\textwidth}
\centering
\includegraphics[width=\linewidth]{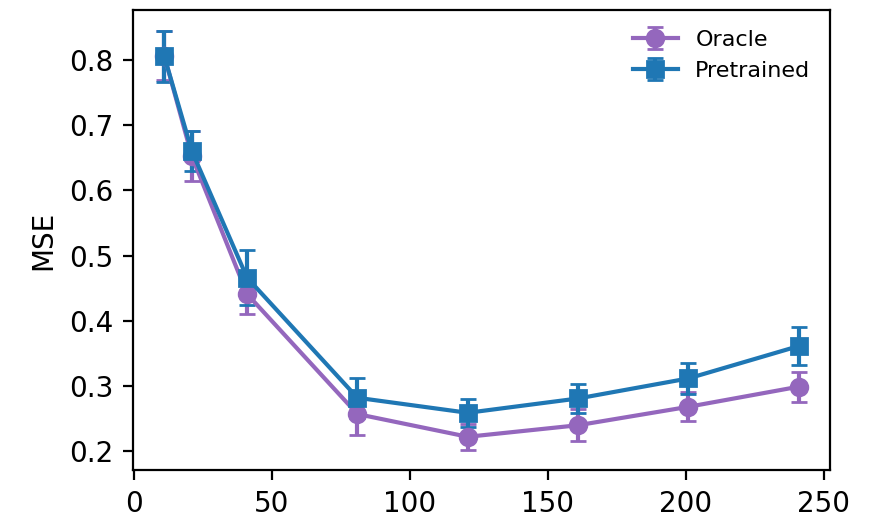}
\caption{MSE vs.\ $W$.}
\label{fig:sim-main-exp3-w-risk}
\end{subfigure}
\caption{End-to-end risk decomposition (Experiment~3). The $n$-, $m$-, and $W$-sweeps probe the supervised estimation, pretraining, and approximation terms of Theorem~\ref{thm:main-end2end}; the oracle uses the true latent coordinates in place of $\widehat{Z}$.}
\label{fig:exp3}
\end{figure}

\subsection{Real-data experiments}
\label{sec:realdata}
We evaluate multi-view pretraining on two KGs: WordNet, whose synset
definitions provide textual side information, and
RENKI-KG~\citep{liu2024representation}, a large-scale medical KG that links
clinical concepts, diagnoses, and medications across curated ontologies and
drug databases. Full protocols and additional results are given in
Appendix~\ref{app:realdata}.

\textbf{Setup and metrics.}
WordNet contains $116{,}566$ entities, $33$ relations, and $373{,}330$
positive triples. RENKI-KG contains $118{,}554$ entities of three types,
$9$ relations, and $698{,}368$ positive triples. Both use random edge
splits,
corrupted negative triples, and
relation-balanced AUROC, accuracy, and F1. Baselines are TransE,
DistMult, ComplEx, and RotatE. In \modelname{}, we compare random
coordinates, single-encoder representations, and multi-view KPCA, using
the same relation-specific head architecture across variants. During supervised learning with text-derived coordinates, only the heads
are trained, while the entity coordinates remain frozen. On WordNet, we use $d=128$ and $50$ epochs
with four pretrained encoders (BERT, RoBERTa, ELECTRA, DeBERTa-v3); the
view weights are selected on validation AUROC from a grid with step
$0.25$ on the simplex, which yields $(1:2:1:0)$. On RENKI-KG, we use
$d=128$, uniform weights over seven BERT-family views (BERT, RoBERTa,
NER-BERT, MentalBERT, SciBERT, BioBERT, ClinicalBERT), and $100$ epochs. Results are averaged over $10$ seeds.

\textbf{Fixed-embedding results.}
\label{subsec:realdata-main-results}
Figures~\ref{fig:wordnet} and~\ref{fig:renki} show a consistent pattern
across both KGs. Frozen text-derived coordinates clearly outperform random coordinates with trainable embeddings (except DeBERTa-v3 on WordNet), which highlights the value of pretraining for knowledge graph learning.
Multi-view KPCA is best overall; in AUROC/accuracy/F1 (WordNet; RENKI-KG), it improves over the best graph-only baseline by $0.1300$/$0.0981$/$0.1617$
and $0.0242$/$0.0544$/$0.0702$, and over the average single view by
$0.0453$/$0.0481$/$0.0531$ and $0.0229$/$0.0285$/$0.0344$, winning on all
$10$ seeds ($p<10^{-7}$). Against the strongest single view (BERT on
WordNet, SciBERT on RENKI-KG), multi-view KPCA gains AUROC on every seed
($+0.0028$ and $+0.0017$, $p<0.01$) and matches it in accuracy and F1.
Thus aggregation recovers the best encoder's performance without having
to identify that encoder in advance.
These modest but consistent gains over single views suggest that the
encoders carry complementary information.

\begin{figure}[t!]
\centering
\includegraphics[width=0.89\textwidth]{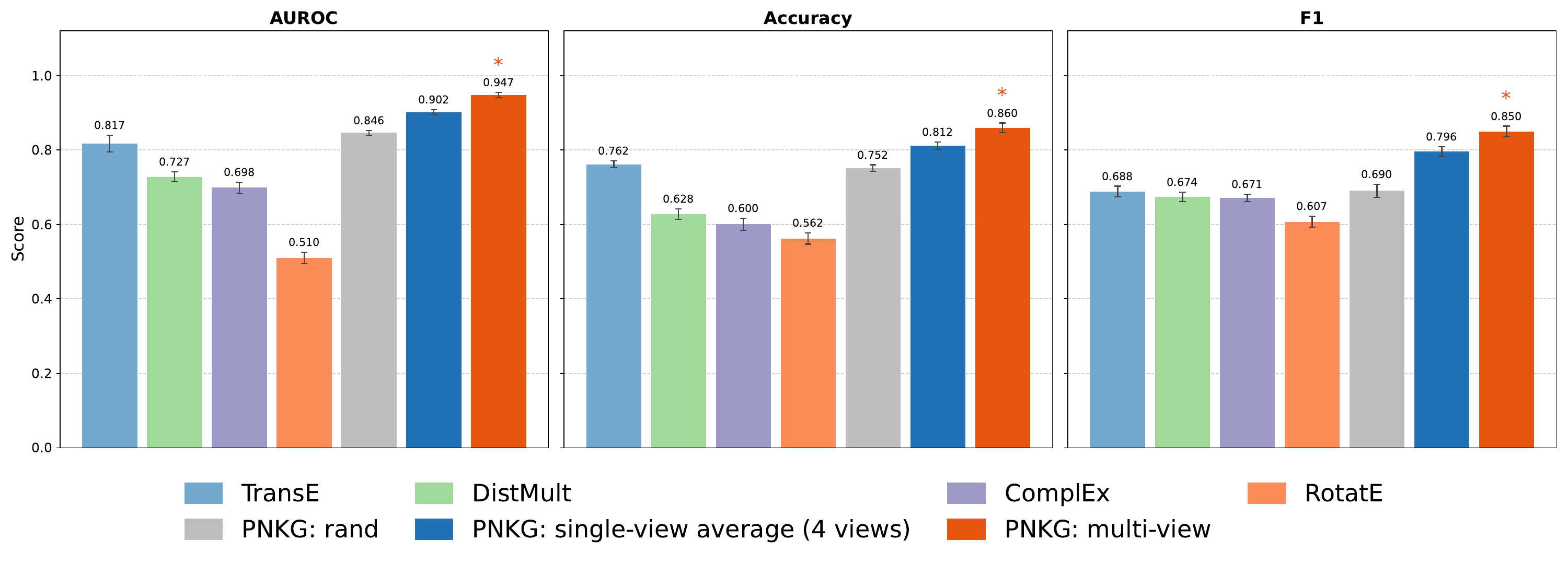}
\caption{Random edge-split results on WordNet with relation-balanced
training and evaluation; text-derived \modelname{} coordinates are fixed, and the random variant trains its embeddings. Results are mean $\pm$
standard deviation over $10$ seeds.}
\label{fig:wordnet}
\end{figure}

\begin{figure}[t!]
\centering
\includegraphics[width=0.9\textwidth]{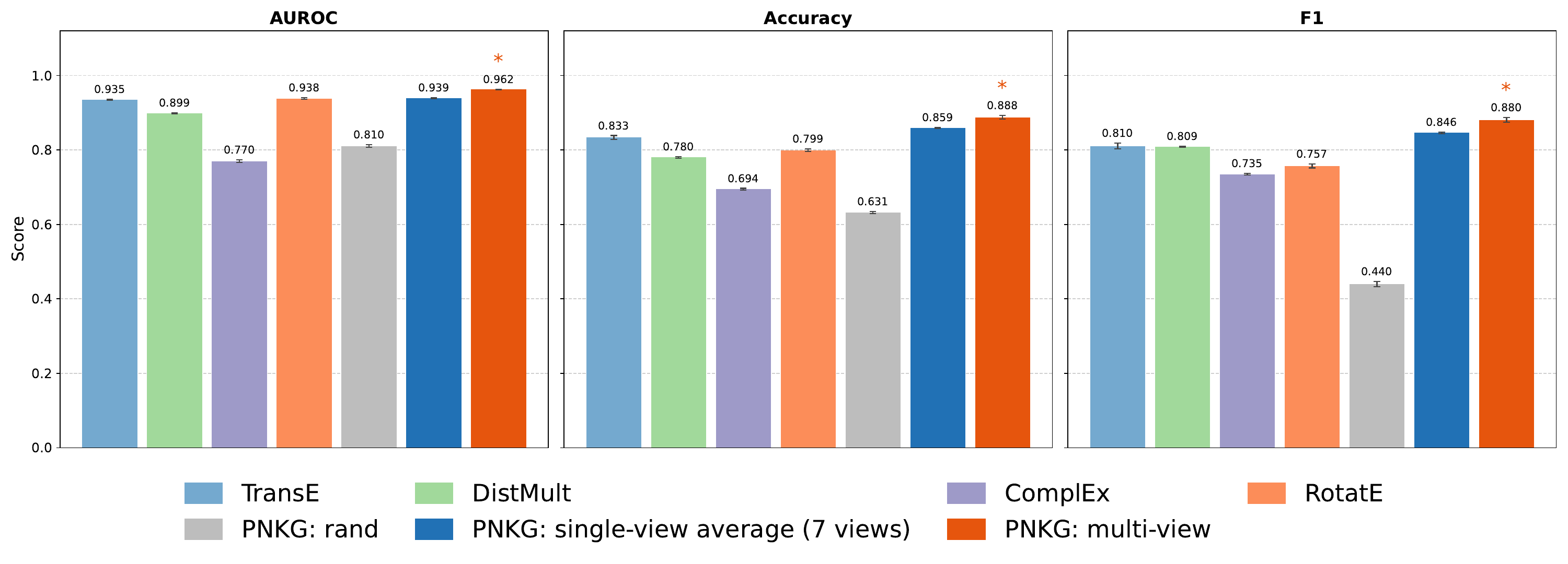}
\caption{Random edge-split results on RENKI-KG with relation-balanced
training and evaluation; text-derived \modelname{} coordinates are fixed, and the random variant trains its embeddings. Results are mean $\pm$
standard deviation over $10$ seeds.}
\label{fig:renki}
\end{figure}

\section{Discussion}
\label{sec:discussion}
\textbf{Interpreting the error decomposition.}
The bound~\eqref{eq:main-end2end-bound} separates four sources of error:
neural approximation, supervised estimation, pretraining, and optimization.
The pretraining term becomes negligible once enough unlabeled data drive it
below the approximation--estimation trade-off. The factor $K$ in the
supervised complexity reflects the cost of learning a separate head for each
relation. The $n$-, $m$-, and $W$-sweeps in Experiment~3 illustrate these
sources of error separately.

\textbf{View weighting.}
The theory favors views with strong effective signal $\alpha_s\gamma_s$ and
small fluctuation scale $v_{s,N,m_s}$. Experiment~2 shows the benefit of
downweighting noisy views, and WordNet confirms it on real data: uniform
weights are pulled down by the weak DeBERTa-v3 view, whereas
validation-selected weights that exclude it give the best results. On
RENKI-KG, where the strongest views are close in quality, uniform weights
already perform best.

\textbf{Limitations and future directions.}
Our analysis assumes frozen pretrained coordinates, an exactly shared signal
subspace~\eqref{eq:shared-subspace-main}, and i.i.d.\ labeled triples.
Allowing fine-tuning, approximately shared subspaces, or unseen test entities
would broaden its scope. Our weighting rule depends on unknown kernel variances and
effective eigengaps, motivating data-driven weighting with guarantees. Other
directions include KG structure as a pretraining view and other relational tasks.

\clearpage
\appendix
\section{Additional Framework Details}
\label{app:framework-details}

\subsection{A linear kernel example}
\label{app:linear-example}
We first specialize \eqref{eq:info-plus-noise-main} to the simplest,
linear isotropic setting. Assume that the noise scale is constant within each
view, $\sigma_{s,j} = \sigma_s$, and that
\[
    \varepsilon_{s,j}\sim N(0,I_{m_s}),\qquad j=1,\ldots,N.
\]
Then
\[
    E_s
    =
    V_s
    +
    \frac{\sigma_s}{\sqrt{m_s}}\xi_s,
\]
where the $j$th rows of $V_s$ and $\xi_s$ are $v_{s,j}^{\top}$ and
$\varepsilon_{s,j}^{\top}$, respectively. The corresponding linear
entity-kernel matrix is
\[
    K_s^{\rm lin}
    =
    \frac{1}{N}E_sE_s^\top
    \in\R^{N\times N}.
\]
With multiple views, we aggregate the entity-kernel matrices as
\[
     K^{\rm lin}(w)
    =
    \sum_{s=1}^S w_s K_s^{\rm lin},
    \qquad
    w\in\Delta_S.
\]
The leading eigenvectors of $K^{\rm lin}(w)$ then provide low-dimensional
coordinates for the common entity set. This is the special case of the general
construction with $k_s(u,u')=u^\top u'$ and $\phi_s(u)=u$, for which
$K_s^{\rm lin}=K_s^{\rm obs}$ and $K^{\rm lin}(w)=K^{\rm obs}(w)$.

\subsection{Weighted direct-sum representation and encoder kernels}
\label{app:direct-sum}

We now show that the weighted multi-view matrix
\(K^{\rm obs}(w)\) in \eqref{eq:weighted-entity-kernel} is itself the Gram
matrix of a single kernel. Write
\(
    \mathbf u_j
    :=
    (u_{1,j},\ldots,u_{S,j})
\)
for the collection of side-information observations of entity
\(j\), and define the weighted direct-sum feature map
\[
    \Phi_w(\mathbf u_j)
    :=
    \bigl(
        \sqrt{w_1}\phi_1(u_{1,j}),
        \ldots,
        \sqrt{w_S}\phi_S(u_{S,j})
    \bigr)
    \in
    \bigoplus_{s=1}^S\mathcal H_s.
\]
Here
\(
    \bigoplus_{s=1}^S\mathcal H_s
\)
is the Hilbert direct sum of the view-specific feature spaces. Since \(S\)
is finite, it consists of the tuples
\(
    (h_1,\ldots,h_S)
\)
with \(h_s\in\mathcal H_s\), equipped with the inner product
\[
    \left\langle
        (h_1,\ldots,h_S),
        (h_1',\ldots,h_S')
    \right\rangle_{\oplus_s\mathcal H_s}
    :=
    \sum_{s=1}^S
    \langle h_s,h_s'\rangle_{\mathcal H_s}.
\]
The weighted views therefore induce the aggregated kernel
\[
    k_w(\mathbf u,\mathbf u')
    :=
    \left\langle
        \Phi_w(\mathbf u),
        \Phi_w(\mathbf u')
    \right\rangle_{\oplus_s\mathcal H_s}
    =
    \sum_{s=1}^S
    w_sk_s(u_s,u_s')
\]
for \(\mathbf u=(u_1,\ldots,u_S)\) and \(\mathbf u'=(u_1',\ldots,u_S')\),
which is positive semidefinite because it is defined by an inner product
(equivalently, because \(w_s\ge0\) and each \(k_s\) is positive semidefinite).
Moreover, \(K^{\rm obs}(w)\) is precisely the normalized Gram matrix of
\(k_w\) on the entity set:
\[
\begin{aligned}
    K^{\rm obs}(w)(j,j')
    &=
    \frac1N
    k_w(\mathbf u_j,\mathbf u_{j'})
    =
    \frac1N
    \left\langle
        \Phi_w(\mathbf u_j),
        \Phi_w(\mathbf u_{j'})
    \right\rangle_{\oplus_s\mathcal H_s}
    =
    \sum_{s=1}^S
    w_sK_s^{\rm obs}(j,j').
\end{aligned}
\]

\begin{remark}[Pretrained encoder-induced kernels]
\label{rem:encoder-induced-kernels}
The kernel formulation naturally accommodates frozen pretrained encoders. In
the language-model application, \(u_{s,j}\in\R^{m_s}\) is a vectorized
representation of the textual description of entity \(j\), and the frozen
encoder of view \(s\) supplies the feature map
\(\phi_s:\R^{m_s}\to\R^{m_s'}\), so that \(\mathcal H_s=\R^{m_s'}\).
The encoder thus induces the positive-semidefinite kernel
\(
    k_s(u,u')
    =
    \langle
        \phi_s(u),\phi_s(u')
    \rangle_{\R^{m_s'}},
\)
and the dot product between language-model embeddings is exactly the
Hilbert-space inner product realizing this kernel. Because the procedure
operates only on the \(N\times N\) kernel matrices, different encoders may have
different feature dimensions \(m_s'\). The framework applies to arbitrary
encoder-induced positive-semidefinite kernels, whereas our theory derives
explicit rates only for the linear, Gaussian distance, and nonlinear
inner-product kernel regimes of Appendix~\ref{app:kernel-regimes}.
\end{remark}

\subsection{Shared signal structure and embedding dimension}
\label{app:shared-signal}

Define the pure-signal kernel matrix of view \(s\) by
\[
    K_s^{\rm sig}(j,j')
    :=
    \frac{1}{N}k_s(v_{s,j},v_{s,j'}),
    \qquad j,j'\in[N],
\]
where \(v_{s,j}\) is the latent signal vector in the information-plus-noise
model~\eqref{eq:info-plus-noise-main}. The central structural condition behind
the \modelname{} framework is that the pure-signal kernels share a common rank-\(d\)
entity eigenspace, while allowing view-specific variation outside that
subspace:
\begin{equation}
    K_s^{\rm sig}
    =
    Z_{\rm sig} \Lambda_s Z_{\rm sig}^\top
    +
    R_s^{\rm sig},
    \qquad
    Z_{\rm sig}^\top Z_{\rm sig}=I_d,
    \qquad
    \Lambda_s\in\R^{d\times d},\ \Lambda_s\succeq0 .
    \label{eq:shared-subspace-main}
\end{equation}
Here \(Z_{\rm sig}\in\R^{N\times d}\) spans the shared population entity
eigenspace, \(\Lambda_s\) captures the strength of view \(s\) along this
subspace, and \(R_s^{\rm sig}\) is a symmetric view-specific residual matrix
supported on the orthogonal complement of the column space of \(Z_{\rm sig}\):
\[
    (R_s^{\rm sig})^\top=R_s^{\rm sig},
    \qquad
    R_s^{\rm sig} Z_{\rm sig}=0 .
\]
This is Assumption~\ref{ass:ek-shared-subspace}.
For a weight vector \(w\in\Delta_S\), the weighted population signal kernel
matrix is then
\begin{equation}
    K^{\rm sig}(w)
    :=
    \sum_{s=1}^S w_s\alpha_s K_s^{\rm sig}
    =
    Z_{\rm sig}
    \left(
        \sum_{s=1}^S w_s\alpha_s\Lambda_s
    \right)
    Z_{\rm sig}^\top
    +
    \sum_{s=1}^S w_s\alpha_sR_s^{\rm sig},
    \label{eq:weighted-signal-kernel}
\end{equation}
where \(\alpha_s>0\) is the signal multiplier of
Appendix~\ref{app:kernel-regimes} (\(\alpha_s=1\) for linear and
inner-product kernels).
The views need not be equally informative; they need only agree on the leading
entity-level signal eigenspace. The observed kernels \(K_s^{\rm obs}\), computed
from the noisy vectors \(u_{s,j}\), are information-plus-noise perturbations of
the rescaled signal kernels \(\alpha_sK_s^{\rm sig}\), where the multiplier
\(\alpha_s\) and the bias and fluctuation scales are given in
Appendix~\ref{app:kernel-regimes}. Since \(\alpha_s>0\), this rescaling does not
change the shared eigenspace \(Z_{\rm sig}\).

\begin{remark}[Choice of $d$]
\label{rem:choice-d}
The embedding dimension \(d\) is the number of leading eigenvectors retained
by Algorithm~\ref{alg:multi-view-kpca}. The theory takes \(d\) to be the rank
of the shared signal eigenspace in~\eqref{eq:shared-subspace-main}. The
pretrained coordinates are identifiable only up to an orthogonal change of
basis, since any such rotation spans the same eigenspace; what matters is the
subspace spanned by the selected eigenvectors, which is why
Theorem~\ref{thm:main-pretrain} is stated in terms of principal angles. In
practice, \(d\) is treated as a tuning parameter: choosing it too small may
discard signal directions, whereas choosing it too large adds residual
directions and unnecessary model complexity.
\end{remark}

\section{Assumptions, Kernel Regimes, and View Weighting}
\label{app:theory-details}
\subsection{Pretraining assumptions}
\setcounter{assumption}{0}
\begin{assumption}[Entity-level information-plus-noise model]
\label{ass:ek-info-noise}
For each view \(s\in[S]\), the observed entity vectors satisfy
\[
    u_{s,j}
    =
    v_{s,j}
    +
    \sigma_{s,j}\frac{\varepsilon_{s,j}}{\sqrt{m_s}},
    \qquad
    j=1,\ldots,N,
\]
where $v_{s,j}\in\R^{m_s}$ is the latent signal vector of entity $j$,
$\varepsilon_{s,j}\in\R^{m_s}$ is high-dimensional noise, and $\sigma_{s,j}$ is a
view-specific noise scale. The normalization $m_s^{-1/2}$ keeps the noise
energy bounded as the feature dimension grows.
\end{assumption}

We next require the noise to satisfy conditions matched to the kernel class
of each view.

\begin{assumption}[Admissible noise regimes]
\label{ass:admissible-noise}
For each view \(s\in[S]\), the noise vectors
\(\{\varepsilon_{s,j}\}_{j=1}^N\) are independent of the signal vectors
\(\{v_{s,j}\}_{j=1}^N\), and one of the following two regimes holds.
\begin{enumerate}
    \item \textbf{Inner-product kernel regime.}
    The noise vectors \(\varepsilon_{s,j}\) are i.i.d.\ across \(j\), with
    \(\E[\varepsilon_{s,j}]=0\), and satisfy the concentration inequality
    \begin{equation}\label{eq:noiseconcentrate}
        \P\{|F(\varepsilon_{s,j})-\E F(\varepsilon_{s,j})|>r\}
        \le
        C_s\exp(-c_s r^{q_s}),
        \qquad r>0,
    \end{equation}
    for every 1-Lipschitz function \(F:\R^{m_s}\to\R\), where
    \(C_s,c_s,q_s>0\) do not depend on \(N\) or \(m_s\). Let
    \[
        \nu_s:=\frac{\E\|\varepsilon_{s,j}\|_2^2}{m_s},
    \]
    and assume \(\nu_s=O(1)\). The noise scales are uniformly bounded,
    \(|\sigma_{s,j}|\le \sigma_{s,\infty}\), and the signal inner products
    are controlled: there is a deterministic sequence \(B_{s,N}\) such that
    \[
        \max_{j,j'\in[N]}|v_{s,j}^{\top}v_{s,j'}|
        \le B_{s,N}.
    \]

    \item \textbf{Gaussian distance-kernel regime.}
    The noise is Gaussian with a homogeneous scale within the view:
    \[
        \varepsilon_{s,j}\stackrel{\rm iid}{\sim}N(0,\Psi_{s,m_s}),
        \qquad
        \sigma_{s,j}=\sigma_s,
    \]
    where \(\Psi_{s,m_s}\succeq0\) does not depend on \(j\) and is normalized
    so that \(\operatorname{tr}(\Psi_{s,m_s})/m_s=1\). The signal energy is
    controlled: there are a deterministic center \(a_s\in\R^{m_s}\) and a
    deterministic sequence \(B_{s,N}\) such that
    \[
        \frac1N\sum_{j=1}^N
        \|v_{s,j}-a_s\|_2^{2q}
        \le
        B_{s,N}^q,
        \qquad q=1,2,3.
    \]
\end{enumerate}
\end{assumption}

\begin{assumption}[Shared effective signal eigenspace]
\label{ass:ek-shared-subspace}
The pure-signal kernel matrices \(K_s^{\rm sig}\) of
Appendix~\ref{app:shared-signal} share a common rank-\(d\) entity
eigenspace: there exist
\(Z_{\rm sig}\in\R^{N\times d}\) with \(Z_{\rm sig}^\top Z_{\rm sig}=I_d\),
positive-semidefinite matrices \(\Lambda_s\in\R^{d\times d}\), and symmetric
residual matrices \(R_s^{\rm sig}\in\R^{N\times N}\) such that
\[
    K_s^{\rm sig}
    =
    Z_{\rm sig}\Lambda_sZ_{\rm sig}^\top
    +
    R_s^{\rm sig},
    \qquad
    R_s^{\rm sig} Z_{\rm sig}=0,
    \qquad s\in[S].
\]
\end{assumption}

\begin{assumption}[Effective signal separation]
\label{ass:ek-eigengap}
Every view has a positive separation margin:
\[
    \gamma_s
    :=
    \lambda_{\min}(\Lambda_s)
    -
    \|R_s^{\rm sig}\|_{\op}
    >0,
    \qquad s\in[S].
\]
\end{assumption}
Since \(\lambda_{\min}\) is superadditive and \(\|\cdot\|_{\op}\) is
subadditive, the weighted gap~\eqref{eq:theory-effective-gap} satisfies
\begin{equation}
    \gamma(w)
    \ge
    \sum_{s=1}^S w_s\alpha_s\gamma_s
    >0
    \qquad\text{for every } w\in\Delta_S .
    \label{eq:gap-lower-bound}
\end{equation}

\begin{assumption}[Sample independence]
\label{ass:pretrain-supervised-independence}
The signal vectors \(\{v_{s,j}\}\) are deterministic (equivalently, all
pretraining statements are conditional on them), and the noise arrays
\(\{\varepsilon_{1,j}\}_{j=1}^N,\ldots,\{\varepsilon_{S,j}\}_{j=1}^N\) are
independent across views, so that the entity matrices \(E_1,\ldots,E_S\)
used for pretraining are independent. Whenever a supervised sample
\(\mathcal D_{\rm sup}=\{(X_i,Y_i)\}_{i=1}^n\) is also considered, it is
independent of the pretraining matrices.
\end{assumption}

Together,
Assumptions~\ref{ass:ek-info-noise}--\ref{ass:pretrain-supervised-independence}
ensure that the weighted population kernel
\(\sum_{s=1}^S w_s\alpha_sK_s^{\rm sig}\) has a well-separated rank-\(d\)
eigenspace, and that each observed kernel \(K_s^{\rm obs}\) is a controlled
information-plus-noise perturbation of its population counterpart
\(\alpha_sK_s^{\rm sig}\), for any mixture of linear, Gaussian distance, and
nonlinear inner-product kernels across views.
Theorem~\ref{thm:main-pretrain} turns this structure into a nonasymptotic
Davis--Kahan-type bound on the principal-angle distance between the aggregated
empirical eigenspace \(\widehat Z\) and the signal eigenspace \(Z_{\rm sig}\).
The bound is expressed through view-specific bias scales \(b_{s,N,m_s}\),
diagonal-correction scales \(d_{s,N,m_s}\), and fluctuation scales
\(v_{s,N,m_s}\), which we calibrate separately for each kernel regime in
Appendix~\ref{app:kernel-regimes}.
Although we use the theorem to control the pretraining error of the \modelname{}
pipeline, its statement makes no reference to the downstream supervised task.
It is a general finite-sample subspace-recovery guarantee for weighted
multi-view (equivalently, multi-kernel) PCA under heterogeneous
information-plus-noise models, and may be of independent interest beyond
knowledge graphs.

\subsection{Supervised-stage assumptions}
\label{app:supervised-assumptions}
Theorem~\ref{thm:main-end2end} adds three conditions to the pretraining
assumptions above: a sampling condition on the labeled triples, a smooth
latent-score model for the KG scoring function, and an alignment condition
that links the pretrained eigenspace to the latent coordinates of that model.
These conditions control, respectively, supervised generalization, neural
approximation, and coordinate transfer together with identifiability.
Throughout, $\sigma_y^2$ denotes the sub-Gaussian variance proxy of the
\emph{label noise}.

\begin{assumption}[Supervised sampling]
\label{ass:sup-sampling-app}
Conditionally on $X_1,\ldots,X_n$, the noises $\eps_i$ in
\eqref{eq:supervised-label-model} are independent and sub-Gaussian with
variance proxy $\sigma_y^2$. The target is bounded:
$\sup_{x\in\cX}|f^\star(x)|\le B$. For out-of-sample bounds,
$X_1,\ldots,X_n$ are i.i.d.\ from a distribution $P$ on $\cX$ that is
uniform on entities: if $X=(h,r,t)\sim P$, then $h$ and $t$ are each
uniformly distributed on $[N]$.
\end{assumption}

\begin{assumption}[Smooth latent KG model]
\label{ass:smooth-kg-app}
There exist latent entity coordinates
$Z^\star=(z_1^\star,\ldots,z_N^\star)^\top\in\R^{N\times d}$,
a radius $R_\star>1$ with $\max_{j\in[N]}\|z_j^\star\|_2\le R_\star-1$,
an integer smoothness index $\beta\ge1$, and bounded
relation-specific functions $g_r^\star:\mathcal D_\star\to[-B,B]$,
where $\mathcal D_\star=[-R_\star,R_\star]^d\times[-R_\star,R_\star]^d$,
such that
\begin{equation}
f^\star(h,r,t)=g_r^\star(z_h^\star,z_t^\star),\qquad
\|g_r^\star\|_{W^{\beta,\infty}(\mathcal D_\star)}\le S_\beta.
\label{eq:theory-smoothness}
\end{equation}
Here $W^{\beta,\infty}(\mathcal D_\star)$ denotes the Sobolev space of
functions on $\mathcal D_\star$ whose weak derivatives up to order $\beta$
are essentially bounded, with norm
$\|g\|_{W^{\beta,\infty}(\mathcal D_\star)}
:=\max_{|\alpha|\le\beta}\|D^\alpha g\|_{L^\infty(\mathcal D_\star)}$,
and $S_\beta$ is a common bound on the norms.
\end{assumption}
The smoothness index $\beta$ governs the ReLU approximation rate. The unit
margin between the latent coordinates and the boundary of \(\mathcal D_\star\)
lets the proof evaluate \(g_r^\star\) at slightly perturbed coordinates; see
Appendix~\ref{app:proof-end2end}.

\begin{assumption}[KPCA alignment with the true latent span]
\label{ass:alignment-app}
There exists an invertible $d\times d$ matrix $A_\star$ such that
\begin{equation}
\sqrt N\,Z_{\rm sig}=Z^\star A_\star,
\qquad \|A_\star\|_{\op}\le\kappa_A,
\qquad \|A_\star^{-1}\|_{\op}\le\kappa_A.
\label{eq:theory-alignment}
\end{equation}
\end{assumption}
The factor \(\sqrt N\) fixes the scale. Because
\(Z_{\rm sig}^\top Z_{\rm sig}=I_d\), the rows of \(Z_{\rm sig}\) have average
squared norm \(d/N\), whereas the rows of \(\sqrt N Z_{\rm sig}\) have average
squared norm \(d\). Hence the latent coordinates \(z_j^\star\) remain of order
one as \(N\) grows, and the scores \(g_r^\star(z_h^\star,z_t^\star)\) do not
degenerate to a relation-specific constant.
This \emph{transfer assumption} requires the latent subspace relevant for KG
prediction to agree with the population subspace recovered by pretraining,
up to a well-conditioned change of coordinates; reducing the pretraining noise
alone does not imply it. The condition is nevertheless broad: it allows
different feature mappings across datasets, and it requires only
subspace-level compatibility between the population pretraining geometry and
the downstream KG geometry, not an exact semantic match between views and
relations or a correctly specified parametric KG scoring function. The matrix
\(A_\star\) accounts for the usual change-of-basis non-identifiability of
spectral representations; each relation-wise ReLU head can absorb it, together
with the factor \(\sqrt N\), into its first affine layer, as shown in
Appendix~\ref{app:proof-end2end}. In particular, Algorithm~\ref{alg:multi-view-kpca}
needs no rescaling of \(\widehat Z\).

\subsection{Kernel-specific scales for Theorem~\ref{thm:main-pretrain}}
\label{app:kernel-regimes}
The sets \(\mathcal S_{\rm lin}\), \(\mathcal S_{\rm gauss}\), and
\(\mathcal S_{\rm ip}\) of linear, Gaussian distance-kernel, and nonlinear
inner-product views partition \([S]\). For each view \(s\in[S]\), we define the
signal multiplier \(\alpha_s\), the deterministic Taylor-bias scale
\(b_{s,N,m_s}\), the nonconstant diagonal-correction scale \(d_{s,N,m_s}\), and
the centered fluctuation scale \(v_{s,N,m_s}\) according to the regime of the
view. Throughout, \(C_{s,b},C_{s,v}>0\) denote constants that depend only on
the view-level noise constants in Assumption~\ref{ass:admissible-noise}, and
not on \(N\) or \(m_s\).

\begin{enumerate}
    \item \textbf{Linear kernel regime.}
    Let \(s\in\mathcal S_{\rm lin}\), with
    \[
        k_s(u,u')=u^\top u',
    \]
    and suppose that the noise satisfies
    Assumption~\ref{ass:admissible-noise}(1) with a homogeneous noise scale,
    \(\sigma_{s,j}\equiv\sigma_s\) for \(j\in[N]\). Set
    \[
        \alpha_s:=1,
        \qquad
        b_{s,N,m_s}^{\rm lin}:=0,
        \qquad
        d_{s,N,m_s}^{\rm lin}:=0,
    \]
    \[
        \left(v_{s,N,m_s}^{\rm lin}\right)^2
        :=
        C_{s,v}
        \left[
            \frac{\sigma_s^2B_{s,N}}{m_s}
            +
            \frac{\sigma_s^4(1+\nu_s)}{m_s}
        \right].
    \]

    \item \textbf{Gaussian distance-kernel regime.}
    Let \(s\in\mathcal S_{\rm gauss}\), with
    \[
        k_s(u,u')
        =
        \exp\{-\tau_s\|u-u'\|_2^2\},
        \qquad
        \tau_s>0,
    \]
    and suppose that the noise satisfies
    Assumption~\ref{ass:admissible-noise}(2). Define
    \[
        \alpha_s
        :=
        \exp(-2\tau_s\sigma_s^2),
        \qquad
        \eta_s^2
        :=
        \frac{
            \sigma_s^4
            \operatorname{tr}(\Psi_{s,m_s}^2)
        }{m_s^2}
        +
        \frac{
            \sigma_s^2
            \|\Psi_{s,m_s}\|_{\op}
            B_{s,N}
        }{m_s},
    \]
    and set
    \[
        b_{s,N,m_s}^{\rm gauss}
        :=
        C_{s,b}
        \left(
            \tau_s^2\eta_s^2
            +
            \tau_s^3\eta_s^3
        \right),
        \qquad
        d_{s,N,m_s}^{\rm gauss}:=0,
    \]
    \[
        \left(v_{s,N,m_s}^{\rm gauss}\right)^2
        :=
        C_{s,v}
        \left(
            \tau_s^2\eta_s^2
            +
            \tau_s^4\eta_s^4
            +
            \tau_s^6\eta_s^6
        \right).
    \]

    \item \textbf{Nonlinear inner-product kernel regime.}
    Let \(s\in\mathcal S_{\rm ip}\), with
    \[
        k_s(u,u')
        =
        f_s(u^\top u'),
    \]
    where \(f_s\in C^3(\R)\) has bounded derivatives,
    \(
        L_{s,r}
        :=
        \sup_{t\in\R}
        |f_s^{(r)}(t)|
        <
        \infty
    \)
    for \(r=1,2,3\), and suppose that the noise satisfies
    Assumption~\ref{ass:admissible-noise}(1). Set
    \[
        \alpha_s:=1,
        \qquad
        \eta_s
        :=
        \sigma_{s,\infty}
        \left(
            \frac{B_{s,N}}{m_s}
        \right)^{1/2}
        +
        \sigma_{s,\infty}^2
        \left(
            \frac{1+\nu_s}{m_s}
        \right)^{1/2}.
    \]
    The Taylor-bias, diagonal-correction, and fluctuation scales are
    \[
        b_{s,N,m_s}^{\rm ip}
        :=
        C_{s,b}
        \left(
            L_{s,2}\eta_s^2
            +
            L_{s,3}\eta_s^3
        \right),
        \qquad
        d_{s,N,m_s}^{\rm ip}
        :=
        \frac{C_{s,b}L_{s,1}\nu_s\sigma_{s,\infty}^2}{\sqrt N},
    \]
    \[
        \left(v_{s,N,m_s}^{\rm ip}\right)^2
        :=
        C_{s,v}
        \left(
            L_{s,1}^2\eta_s^2
            +
            L_{s,2}^2\eta_s^4
            +
            L_{s,3}^2\eta_s^6
        \right).
    \]
\end{enumerate}

For each \(s\in[S]\), the scales entering Theorem~\ref{thm:main-pretrain} are
those of the regime of view \(s\):
\[
    (b_{s,N,m_s},d_{s,N,m_s},v_{s,N,m_s})
    :=
    \begin{cases}
        \left(
            b_{s,N,m_s}^{\rm lin},
            d_{s,N,m_s}^{\rm lin},
            v_{s,N,m_s}^{\rm lin}
        \right),
        & s\in\mathcal S_{\rm lin},\\[2mm]
        \left(
            b_{s,N,m_s}^{\rm gauss},
            d_{s,N,m_s}^{\rm gauss},
            v_{s,N,m_s}^{\rm gauss}
        \right),
        & s\in\mathcal S_{\rm gauss},\\[2mm]
        \left(
            b_{s,N,m_s}^{\rm ip},
            d_{s,N,m_s}^{\rm ip},
            v_{s,N,m_s}^{\rm ip}
        \right),
        & s\in\mathcal S_{\rm ip}.
    \end{cases}
\]

\subsection{Rates and view weighting}
\label{app:weighting}
Let \(M:=\sum_{s=1}^S m_s\), and abbreviate \(b_s:=b_{s,N,m_s}\) and
\(v_s:=v_{s,N,m_s}\). Throughout this subsection we treat \(S\) and the
view-level parameters as fixed, assume \(B_{s,N}=O(1)\) and, for Gaussian
distance-kernel views, \(\|\Psi_{s,m_s}\|_{\op}=O(1)\), and assume that the
effective gap \(\gamma(w)\) is bounded away from zero.

\textbf{Sample-size weighting.}
We first set aside the diagonal correction \(d_{s,N,m_s}\), which arises only
for nonlinear inner-product kernels. Under the above conditions, the scales
of Appendix~\ref{app:kernel-regimes} satisfy
\[
    v_s^2\le \frac{c_s}{m_s},
    \qquad
    b_s=O(m_s^{-1}),
\]
for view-dependent constants \(c_s>0\), so the squared bias is of lower order
than the variance. A natural rule weights each view in proportion to its
sample size, \(w_s^{\rm size}:=m_s/M\). Under this choice,
\[
    \sum_{s=1}^S (w_s^{\rm size})^2v_s^2
    =
    O(M^{-1}),
    \qquad
    \left(
        \sum_{s=1}^S w_s^{\rm size}b_s
    \right)^2
    =
    O(M^{-2}).
\]
The variance term therefore dominates, and at a fixed confidence level
Theorem~\ref{thm:main-pretrain} gives
\(
    \|\sin\Theta(\widehat Z,Z_{\rm sig})\|_F^2
    =
    O(M^{-1})
\).

\textbf{Signal-adjusted inverse-variance weighting.}
A more refined rule accounts for both the noise level and the effective
signal strength of each view through
\(g_s=\alpha_s\gamma_s\) in~\eqref{eq:theory-per-view-gap}. For a Gaussian
distance-kernel view,
\(
    g_s
    =
    e^{-2\tau_s\sigma_s^2}\gamma_s
\),
while \(v_s^2\) contains terms proportional to
\(\tau_s^2\), \(\tau_s^4\), and \(\tau_s^6\). An overly sharp Gaussian
kernel thus both attenuates the effective signal and inflates the
stochastic error.

Ignoring the lower-order bias term and replacing \(\gamma(w)\) by its lower
bound \(\sum_{s}w_sg_s\) from~\eqref{eq:gap-lower-bound} leads to the
variance-only surrogate
\[
    \frac{
        \sum_{s=1}^S w_s^2v_s^2
    }{
        \left(\sum_{s=1}^S w_sg_s\right)^2
    }
    \ \ge\
    \left(
        \sum_{s=1}^S
        \frac{g_s^2}{v_s^2}
    \right)^{-1},
\]
where the inequality is Cauchy--Schwarz. Equality holds for the
\emph{signal-adjusted inverse-variance weights}
\[
    w_s^{\rm IV}
    :=
    \frac{g_s/v_s^2}
    {\sum_{t=1}^S g_t/v_t^2},
\]
which therefore minimize the surrogate over \(\Delta_S\) and give
\[
    \|\sin\Theta(\widehat Z,Z_{\rm sig})\|_F^2
    \lesssim
    \frac1\delta
    \left(
        \sum_{s=1}^S
        \frac{g_s^2}{v_s^2}
    \right)^{-1}
\]
up to the lower-order bias term. The quantity
\(\sum_{s=1}^S g_s^2/v_s^2\) thus acts as the aggregate effective information
across views. If \(v_s^2\asymp c_s/m_s\) and the ratios \(g_s/c_s\) are
comparable across views, then \(w_s^{\rm IV}\asymp m_s/M\). Both rules
therefore achieve the \(O(M^{-1})\) rate, and the inverse-variance rule has a
leading constant that is no larger, since it minimizes the surrogate.

\textbf{Diagonal correction.}
For nonlinear inner-product kernels, the nonconstant diagonal correction
contributes \(O(N^{-1})\) to the squared subspace error, so the overall rate
becomes \(O(M^{-1}+N^{-1})\). This term is absent in the linear regime with a
homogeneous noise scale, because the diagonal correction is then a multiple
of the identity, which can be absorbed into the target without changing its
eigenspace; it is likewise absent for Gaussian distance kernels, whose
diagonal is constant. Diagonal-imputation methods for heteroskedastic
PCA~\citep{yan2024inference} suggest a possible route for estimating and
removing the nonconstant correction.

\textbf{Why the linear regime matters.}
We highlight the linear regime because it is a natural leading-order
approximation to high-dimensional random kernel matrices:
\citet{el2010spectrum} shows that the leading spectral behavior of many
commonly used kernel matrices in high dimensions is governed by
covariance-type linear matrices.

\section{Proofs}
\label{app:proofs}

\subsection{Proof of Theorem~\ref{thm:main-pretrain} (subspace recovery)}

We first prove view-level kernel approximation bounds for nonlinear
inner-product kernels, then for linear kernels, and finally for Gaussian
distance kernels. The linear kernel corresponds to \(f_s(t)=t\), but we treat
it separately because, under a homogeneous noise scale, its diagonal
correction is a multiple of \(I_N\) and can be absorbed into the target
without changing its eigenspace. Throughout, kernel matrices are normalized by \(1/N\),
as in the main text. For each view \(s\), recall that
\[
    \alpha_s
    :=
    \begin{cases}
    1, & s\in\mathcal S_{\rm ip}\cup\mathcal S_{\rm lin},\\
    \exp(-2\tau_s\sigma_s^2), & s\in\mathcal S_{\rm gauss},
    \end{cases}
\]
where \(\mathcal S_{\rm ip}\), \(\mathcal S_{\rm lin}\), and
\(\mathcal S_{\rm gauss}\) denote the sets of nonlinear inner-product, linear,
and Gaussian distance-kernel views, respectively; these three sets partition
\([S]\). We also define the ``harmless identity-shift''
coefficient
\[
    \chi_s
    :=
    \begin{cases}
        \nu_s\sigma_s^2/N,
        & s\in\mathcal S_{\rm lin},\\
        0,
        & s\in\mathcal S_{\rm ip},\\
        (1-\alpha_s)/N,
        & s\in\mathcal S_{\rm gauss}.
    \end{cases}
\]

\begin{lemma}[Inner-product kernel decomposition]
\label{lem:ip-kernel-approx}
Let \(s\in\mathcal S_{\rm ip}\), with
\(
    k_s(u,u')=f_s(u^\top u')
\),
and suppose that \(f_s\) satisfies the derivative bounds of the nonlinear
inner-product regime in Appendix~\ref{app:kernel-regimes} and that the noise
satisfies Assumption~\ref{ass:admissible-noise}(1). Let
\[
    \eta_s
    :=
    \sigma_{s,\infty}
    \left(\frac{B_{s,N}}{m_s}\right)^{1/2}
    +
    \sigma_{s,\infty}^2\sqrt{\frac{1+\nu_s}{m_s}}.
\]
Then
\[
    K_s^{\rm obs}-K_s^{\rm sig}
    = \Upsilon_s + \Xi_s,
    \qquad
    \mathbb E\Xi_s=0,
\]
where
\[
    \|\Upsilon_s\|_F
    \le b_s+d_s,
    \qquad
    \mathbb E\|\Xi_s\|_F^2
    \le v_s^2,
\]
with
\[
    b_s
    :=
    C_{s,b}
    \left(
        L_{s,2}\eta_s^2
        +
        L_{s,3}\eta_s^3
    \right),
    \qquad
    d_s:=C_{s,b}\frac{L_{s,1}\nu_s\sigma_{s,\infty}^2 }{\sqrt N},
\]
and
\[
    v_s^2
    :=
    C_{s,v}
    \left(
        L_{s,1}^2\eta_s^2
        +
        L_{s,2}^2\eta_s^4
        +
        L_{s,3}^2\eta_s^6
    \right).
\]
The constants \(C_{s,b}\) and \(C_{s,v}\) depend only on
\((C_s,c_s,q_s)\) and a uniform upper bound on \(\nu_s\), and not on \(N\)
or \(m_s\).
\end{lemma}

\begin{proof}
We suppress the view index \(s\) and write
\begin{gather*}
    m=m_s,\qquad
    \nu=\nu_s,\qquad
    \sigma_j=\sigma_{s,j},\qquad
    \sigma_\infty=\sigma_{s,\infty},\\
    B_N=B_{s,N},
    \qquad
    f=f_s,
    \qquad
    L_r=L_{s,r},\\
    u_j=u_{s,j},
    \qquad
    v_j=v_{s,j},
    \qquad
    \varepsilon_j=\varepsilon_{s,j}.
\end{gather*}
The signal vectors are fixed throughout, and all expectations and
probabilities are taken over the noise only. In this proof, \(C\) denotes a
constant that depends only on \((C_s,c_s,q_s)\) and an upper bound on
\(\nu\), and whose value may change from line to line; \(C_{s,p}\) denotes
such a constant that may also depend on \(p\).

For \(j,k\in[N]\), define
\[
    x_{jk}
    :=
    v_j^\top v_k
    +
    \mathbf 1_{\{j=k\}}\nu\sigma_j^2,
    \qquad
    \Delta_{jk}
    :=
    u_j^\top u_k-x_{jk}.
\]
Since
\(
    u_j
    =
    v_j+\sigma_j\varepsilon_j/\sqrt m
\),
expanding the inner product gives
\[
    \Delta_{jk}
    =
    \frac{\sigma_j}{\sqrt m}
    \varepsilon_j^\top v_k+
    \frac{\sigma_k}{\sqrt m}
    \varepsilon_k^\top v_j+
    \frac{\sigma_j\sigma_k}{m}
    \varepsilon_j^\top\varepsilon_k
    -
    \mathbf 1_{\{j=k\}}\nu\sigma_j^2.
\]
For \(j\ne k\), every term is centered because the noise vectors are
centered and independent across entities; for \(j=k\),
\(\mathbb E\|\varepsilon_j\|_2^2=m\nu\) by the definition of \(\nu\), so the
last two terms cancel in expectation. Hence
\(
    \mathbb E\Delta_{jk}=0
\).

\textit{Moment bounds for \(\Delta_{jk}\).}
The noise assumption implies the following standard fact: if
\(G:\mathbb R^{m}\to\mathbb R\) is \(L\)-Lipschitz, then for every fixed
\(p\ge1\),
\[
    \|G(\varepsilon_j)-\mathbb EG(\varepsilon_j)\|_{L^p}
    \le
    C_{s,p}L.
\]
Indeed, let \(T:=G(\varepsilon_j)-\mathbb E G(\varepsilon_j)\). Applying
\eqref{eq:noiseconcentrate} to the \(1\)-Lipschitz function \(G/L\) gives
\[
    \mathbb P(|T|>t)
    \le
    C_s \exp\left\{ -c_s\left(\frac{t}{L}\right)^{q_s} \right\},
\]
and the tail-integral identity yields
\[
\begin{aligned}
    \mathbb E|T|^p
    &=
    p\int_0^\infty t^{p-1}\mathbb P(|T|>t)\,dt\\
    &\le
    pC_s \int_0^\infty t^{p-1}
    \exp\left\{ -c_s\left(\frac{t}{L}\right)^{q_s} \right\}dt\\
    &=
    \frac{pC_s}{q_s} c_s^{-p/q_s}
    \Gamma\left(\frac{p}{q_s}\right)L^p.
\end{aligned}
\]
Hence \(\|T\|_{L^p}\le C_{s,p}L\), where \(C_{s,p}\) depends only on
\((p,C_s,c_s,q_s)\) and not on \(N\) or \(m\).

The same concentration property tensorizes, up to constants, to an
independent pair \((\varepsilon_j,\varepsilon_k)\). To see this, let
\(G:\mathbb R^m\times\mathbb R^m\to\mathbb R\) be \(1\)-Lipschitz under the
Euclidean product metric, and define
\(
    g(x):=\mathbb E_{\varepsilon_k}G(x,\varepsilon_k)
\).
Then \(g\) is \(1\)-Lipschitz, and conditioning on \(\varepsilon_j\) gives
\[
\begin{aligned}
    \mathbb P
    \{|G-\mathbb EG|>r\}
    &\le
    \mathbb P
    \{|G-g(\varepsilon_j)|>r/2\}
    +
    \mathbb P
    \{|g(\varepsilon_j)-\mathbb Eg(\varepsilon_j)|>r/2\}\\
    &\le
    2C_s
    \exp(-c_s2^{-q_s}r^{q_s}).
\end{aligned}
\]

Consider first the signal--noise term
\(
    \sigma_j\varepsilon_j^\top v_k/\sqrt m
\).
As a function of \(\varepsilon_j\), it is centered and Lipschitz with
constant
\[
    \frac{|\sigma_j|\|v_k\|_2}{\sqrt m}
    \le
    \sigma_\infty
    \left(\frac{B_N}{m}\right)^{1/2},
\]
so for every fixed \(p\ge1\),
\[
    \left\|\frac{\sigma_j}{\sqrt m}\varepsilon_j^\top v_k\right\|_{L^p}
    \le
    C_{s,p}
    \sigma_\infty
    \left(\frac{B_N}{m}\right)^{1/2}.
\]
The same bound holds for the other signal--noise term.

We next control the noise--noise term. For \(j=k\), let
\(
    R_j:=|\sigma_j|\|\varepsilon_j\|_2
\).
Since \(R_j\) is \(|\sigma_j|\)-Lipschitz in \(\varepsilon_j\),
\[
    \|R_j-\mathbb ER_j\|_{L^p}
    \le
    C_{s,p}|\sigma_j|.
\]
Moreover,
\[
    \mathbb ER_j^2
    =
    m\nu\sigma_j^2,
    \qquad
    \mathbb ER_j
    \le
    |\sigma_j|\sqrt{m\nu}.
\]
Using the identity
\[
    R_j^2-\mathbb ER_j^2
    =
    (R_j-\mathbb ER_j)(R_j+\mathbb ER_j)
    -
    \operatorname{Var}(R_j),
\]
Hölder's inequality, and
\(\operatorname{Var}(R_j)\le C_{s,2}^2\sigma_j^2\), we obtain for every fixed
\(p\ge1\)
\begin{align*}
    \|R_j^2-\mathbb ER_j^2\|_{L^p}
    &\le
    \|(R_j-\mathbb ER_j)(R_j+\mathbb ER_j)\|_{L^p}+\operatorname{Var}(R_j)\\
    &\le
    \|R_j-\mathbb ER_j\|_{L^{2p}}\|R_j+\mathbb ER_j\|_{L^{2p}}
    +\operatorname{Var}(R_j)\\
    &\le
    C_{s,p}|\sigma_j|\bigl(C_{s,p}|\sigma_j|+2|\sigma_j|\sqrt{m\nu}\bigr)\\
    &\le
    C_{s,p}\sigma_\infty^2(1+\sqrt{m\nu}).
\end{align*}
Consequently,
\[
    \left\|
        \sigma_j^2
        \left(
            \frac{\|\varepsilon_j\|_2^2}{m}-\nu
        \right)
    \right\|_{L^p}
    \le
    C_{s,p}\sigma_\infty^2\left(
    \frac{1}{m}+\sqrt{\frac{\nu}{m}}\right).
\]

For \(j\ne k\), define
\[
    R_{jk}^{\pm}
    :=
    \|\sigma_j\varepsilon_j
    \pm
    \sigma_k\varepsilon_k\|_2.
\]
As functions of the independent pair \((\varepsilon_j,\varepsilon_k)\), these
are Lipschitz with constants at most
\(
    (\sigma_j^2+\sigma_k^2)^{1/2}
    \le
    \sqrt2\,\sigma_\infty
\),
and
\[
    \mathbb E(R_{jk}^{+})^2
    =
    \mathbb E(R_{jk}^{-})^2
    =
    m\nu(\sigma_j^2+\sigma_k^2).
\]
The preceding squared-norm argument, combined with the tensorized
concentration bound, therefore gives
\[
    \left\|
        (R_{jk}^{\pm})^2
        -
        \mathbb E(R_{jk}^{\pm})^2
    \right\|_{L^p}
    \le
    C_{s,p}\sigma_\infty^2(1+\sqrt {m\nu}).
\]
By the polarization identity
\(
    4\sigma_j\sigma_k
    \varepsilon_j^\top\varepsilon_k
    =
    (R_{jk}^{+})^2-(R_{jk}^{-})^2
\)
and the equality of the two expectations,
\[
    \left\|
        \frac{\sigma_j\sigma_k}{m}
        \varepsilon_j^\top\varepsilon_k
    \right\|_{L^p}
    \le
    C_{s,p}
    \sigma_\infty^2\left(\frac{1}{m}+\sqrt{\frac{\nu}{m}}\right)
    \le
    C_{s,p}
    \sigma_\infty^2\sqrt{\frac{\nu+1}{m}}.
\]

Combining the two signal--noise terms with the noise--noise term gives, for
\(p=2,4,6\) and uniformly over \(j,k\),
\[
    \|\Delta_{jk}\|_{L^p}
    \le
    C_{s,p}\,\omega,
    \qquad
    \omega
    :=
    \sigma_\infty
    \left(\frac{B_N}{m}\right)^{1/2}
    +\sigma_\infty^2\sqrt{\frac{\nu+1}{m}},
\]
so that \(\omega=\eta_s\) after restoring the view index. Equivalently,
\[
    \mathbb E|\Delta_{jk}|^p
    \le
    C_{s,p}\,\omega^p,
    \qquad p=2,4,6.
\]

\textit{Taylor expansion.}
Taylor's theorem with integral remainder~\citep{firey1960remainder} gives
\[
    f(x_{jk}+\Delta_{jk})-f(x_{jk})
    =
    f'(x_{jk})\Delta_{jk}
    +
    \frac12 f''(x_{jk})\Delta_{jk}^2
    +
    \mathcal R_{jk},
\]
where
\[
    \mathcal R_{jk}
    =
    \frac{\Delta_{jk}^3}{2}
    \int_0^1
    (1-r)^2
    f^{(3)}(x_{jk}+r\Delta_{jk})\,dr.
\]
Since \(f^{(3)}\) is bounded by \(L_3\),
\[
    |\mathcal R_{jk}|
    \le
    \frac{L_3}{6}|\Delta_{jk}|^3,
    \qquad\text{hence}\qquad
    \mathbb E|\mathcal R_{jk}|^2
    \le
    CL_3^2\omega^6.
\]
Define
\[
    A^{(1)}(j,k)
    :=
    \frac1N f'(x_{jk})\Delta_{jk},\qquad
    A^{(2)}(j,k)
    :=
    \frac1{2N}f''(x_{jk})\Delta_{jk}^2,
    \qquad
    A^{(R)}(j,k)
    :=
    \frac1N \mathcal R_{jk},
\]
and the effective inner-product kernel matrix
\[
    K^\sharp(j,k)
    :=
    \frac1N
    f\!\left(
        v_{j}^{\top}v_{k}
        +
        \mathbf 1_{\{j=k\}}
        \nu\sigma_{j}^2
    \right)
    =\frac1N f(x_{jk}).
\]
Then
\[
    K^{\rm obs}-K^\sharp
    =
    A^{(1)}+A^{(2)}+A^{(R)}.
\]

\textit{Bias--fluctuation split.}
Define
\[
    \Upsilon'
    :=
    \mathbb E[K^{\rm obs}-K^\sharp],
    \qquad
    \Xi
    :=
    K^{\rm obs}-K^\sharp-\Upsilon'.
\]
By construction, \(\mathbb E\Xi=0\). Since \(\mathbb E\Delta_{jk}=0\), we
have \(\mathbb EA^{(1)}=0\), and hence
\[
    \Upsilon'
    =
    \mathbb EA^{(2)}
    +
    \mathbb EA^{(R)}.
\]
For the second-order bias,
\[
    \|\mathbb EA^{(2)}\|_F^2
    =
    \sum_{j,k}
    \left|
        \frac{f''(x_{jk})}{2N}
        \mathbb E\Delta_{jk}^2
    \right|^2
    \le
    \frac{L_2^2}{4N^2}
    \sum_{j,k}
    \left(
        \mathbb E\Delta_{jk}^2
    \right)^2
    \le
    CL_2^2\omega^4,
\]
so \(\|\mathbb EA^{(2)}\|_F\le CL_2\omega^2\). For the Taylor-remainder bias,
\[
    \|\mathbb EA^{(R)}\|_F^2
    \le
    \frac1{N^2}
    \sum_{j,k}
    \mathbb E|\mathcal R_{jk}|^2
    \le
    CL_3^2\omega^6.
\]
Combining the two estimates gives
\[
    \|\Upsilon'\|_F
    \le
    C_{s,b}
    \left(
        L_2\omega^2
        +
        L_3\omega^3
    \right).
\]

For the centered fluctuation,
\[
    \Xi
    =
    A^{(1)}
    +
    \left(A^{(2)}-\mathbb EA^{(2)}\right)+
    \left(A^{(R)}-\mathbb EA^{(R)}\right).
\]
The three terms satisfy
\[
\begin{aligned}
    \mathbb E\|A^{(1)}\|_F^2
    &=
    \frac1{N^2}
    \sum_{j,k}
    f'(x_{jk})^2
    \mathbb E\Delta_{jk}^2
    \le
    CL_1^2\omega^2,\\
    \mathbb E
    \bigl\|
        A^{(2)}-\mathbb EA^{(2)}
    \bigr\|_F^2
    &\le
    \mathbb E\|A^{(2)}\|_F^2
    \le
    \frac{L_2^2}{4N^2}
    \sum_{j,k}
    \mathbb E|\Delta_{jk}|^4
    \le
    CL_2^2\omega^4,\\
    \mathbb E
    \bigl\|
        A^{(R)}-\mathbb EA^{(R)}
    \bigr\|_F^2
    &\le
    \mathbb E\|A^{(R)}\|_F^2
    \le
    CL_3^2\omega^6.
\end{aligned}
\]
Since
\(
    \|A+B+C\|_F^2
    \le
    3
    (
        \|A\|_F^2+\|B\|_F^2+\|C\|_F^2
    )
\)
for any matrices \(A,B,C\), we conclude that
\[
    \mathbb E\|\Xi\|_F^2
    \le
    C_{s,v}
    \left(
        L_1^2\omega^2
        +
        L_2^2\omega^4
        +
        L_3^2\omega^6
    \right).
\]

\textit{Diagonal correction.}
It remains to compare \(K^\sharp\) with \(K^{\rm sig}\). The two matrices
agree off the diagonal, so their difference is the deterministic diagonal
matrix
\[
    D^{\rm diag}(j,k)
    :=
    \mathbf 1_{\{j=k\}} \frac1N
    \left[
        f\!\left( \|v_{j}\|_2^2+\nu\sigma_{j}^2 \right)
        -
        f(\|v_{j}\|_2^2)
    \right].
\]
By the mean-value theorem,
\[
    |D^{\rm diag}(j,j)|
    =
    \frac1N
    \left|
        f(\|v_j\|_2^2+\nu\sigma_j^2)
        -
        f(\|v_j\|_2^2)
    \right|
    \le
    \frac{L_1\nu\sigma_j^2}{N},
\]
and therefore
\[
    \|D^{\rm diag}\|_F
    \le
    \frac{
        L_1\nu
        \left(\sum_{j=1}^N\sigma_j^4\right)^{1/2}
    }{N}
    \le
    \frac{L_1\nu\sigma_\infty^2}{\sqrt N}.
\]

\textit{Conclusion.}
Restoring the view index, set
\[
    \Upsilon_s:=\Upsilon'+D_s^{\rm diag},
    \qquad
    \Xi_s:=\Xi.
\]
Then
\(
    K_s^{\rm obs}-K_s^{\rm sig}
    =\Upsilon_s+\Xi_s
\)
with \(\mathbb E\Xi_s=0\), and the bounds above give
\(\|\Upsilon_s\|_F\le b_s+d_s\) (taking \(C_{s,b}\ge1\)) and
\(\mathbb E\|\Xi_s\|_F^2\le v_s^2\).
\end{proof}

\begin{lemma}[Linear-kernel specialization]
\label{lem:linear-kernel-approx}
Let \(s\in\mathcal S_{\rm lin}\), so that
\[
    k_s(u,u')=u^\top u',
\]
and suppose that the inner-product noise regime of
Assumption~\ref{ass:admissible-noise}(1) holds. Assume, in addition, that the
noise scale is homogeneous within the view, \(\sigma_{s,j}\equiv\sigma_s\)
for \(j\in[N]\). Define
\[
    \chi_s^{\rm lin}
    :=
    \frac{\nu_s\sigma_s^2}{N},
    \qquad
    \Delta_s^{\rm lin}
    :=
    K_s^{\rm obs}
    -
    \left(
        K_s^{\rm sig}
        +
        \chi_s^{\rm lin}I_N
    \right).
\]
Then
\[
    \Delta_s^{\rm lin}
    =
    \Xi_s,
    \qquad
    \mathbb E\Xi_s=0.
\]
In particular, the deterministic finite-\(m_s\) bias vanishes: in the
per-view decomposition
\(K_s^{\rm obs}-\alpha_sK_s^{\rm sig}-\chi_sI_N=\Upsilon_s+\Xi_s\) used in
the proof of Theorem~\ref{thm:main-pretrain}, with \(\alpha_s=1\) and
\(\chi_s=\chi_s^{\rm lin}\), one may take \(\Upsilon_s=0\), so \(b_s=0\). Moreover,
\[
    \mathbb E
    \left\|
        \Xi_s
    \right\|_F^2
    \le
    v_s^2,
    \qquad
    v_s^2
    :=
    C_{s,v}
    \left[
        \frac{\sigma_s^2B_{s,N}}{m_s}
        +
        \frac{
            \sigma_s^4(1+\nu_s)
        }{m_s}
    \right],
\]
where \(C_{s,v}\) depends only on \((C_s,c_s,q_s)\) and a uniform upper bound
on \(\nu_s\), and not on \(N\) or \(m_s\).
\end{lemma}

\begin{proof}
We suppress the view index \(s\) and write
\[
    m=m_s,
    \qquad
    \sigma=\sigma_s,
    \qquad
    \nu=\nu_s,
    \qquad
    B_N=B_{s,N},
\]
and \(u_j=u_{s,j}\), \(v_j=v_{s,j}\), \(\varepsilon_j=\varepsilon_{s,j}\). As in
the proof of Lemma~\ref{lem:ip-kernel-approx}, \(C\) denotes a constant that
depends only on \((C_s,c_s,q_s)\) and an upper bound on \(\nu\), and whose
value may change from line to line.
The observed entity vectors are
\(
    u_j
    =
    v_j
    +
    \sigma\varepsilon_j/\sqrt m
\),
and for the linear kernel
\[
    K^{\rm obs}(j,k)
    =
    \frac1N u_j^\top u_k,
    \qquad
    K^{\rm sig}(j,k)
    =
    \frac1N v_j^\top v_k.
\]
Expanding the inner product gives
\[
    u_j^\top u_k
    =
    v_j^\top v_k
    +
    \frac{\sigma}{\sqrt m}
    \varepsilon_j^\top v_k
    +
    \frac{\sigma}{\sqrt m}
    \varepsilon_k^\top v_j
    +
    \frac{\sigma^2}{m}
    \varepsilon_j^\top\varepsilon_k .
\]
Define
\[
\begin{aligned}
    \Delta_{jk}
    &:=
    u_j^\top u_k
    -
    v_j^\top v_k
    -
    \mathbf 1_{\{j=k\}}\nu\sigma^2\\
    &=
    \frac{\sigma}{\sqrt m}
    \varepsilon_j^\top v_k
    +
    \frac{\sigma}{\sqrt m}
    \varepsilon_k^\top v_j
    +
    \frac{\sigma^2}{m}
    \left(
        \varepsilon_j^\top\varepsilon_k
        -
        \mathbf 1_{\{j=k\}}m\nu
    \right).
\end{aligned}
\]
We claim that \(\mathbb E\Delta_{jk}=0\). The two signal--noise terms have
zero mean because the noise is centered. If \(j\ne k\), independence and
centering give \(\mathbb E(\varepsilon_j^\top\varepsilon_k)=0\); if \(j=k\),
the definition \(\nu=\mathbb E\|\varepsilon_j\|_2^2/m\) gives
\[
    \mathbb E
    \left[
        \frac{\|\varepsilon_j\|_2^2}{m}-\nu
    \right]
    =
    0.
\]

The moment bounds established in the proof of
Lemma~\ref{lem:ip-kernel-approx} give, uniformly over \(j,k\in[N]\),
\[
    \left\|
        \Delta_{jk}
    \right\|_{L^2}
    \le
    C
    \left[
        \sigma
        \left(\frac{B_N}{m}\right)^{1/2}
        +
        \sigma^2
        \sqrt{\frac{1+\nu}{m}}
    \right],
\]
and hence
\[
    \mathbb E
    \left|
        \Delta_{jk}
    \right|^2
    \le
    C
    \left[
        \frac{\sigma^2B_N}{m}
        +
        \frac{\sigma^4(1+\nu)}{m}
    \right].
\]

The residual matrix satisfies
\[
    \Delta^{\rm lin}(j,k)
    =
    K^{\rm obs}(j,k)
    -
    K^{\rm sig}(j,k)
    -
    \frac{\nu\sigma^2}{N}
    \mathbf 1_{\{j=k\}}
    =
    \frac1N
    \Delta_{jk}.
\]
Thus \(\mathbb E \Delta^{\rm lin}=0\), so we may set \(\Xi:=\Delta^{\rm lin}\) and
\(\Upsilon:=0\); in particular \(b_s=0\). Furthermore,
\[
    \mathbb E
    \left\|
        \Xi
    \right\|_F^2
    =
    \frac1{N^2}
    \sum_{j,k=1}^N
    \mathbb E
    \left|
        \Delta_{jk}
    \right|^2
    \le
    C_{s,v}
    \left[
        \frac{\sigma^2B_N}{m}
        +
        \frac{\sigma^4(1+\nu)}{m}
    \right].
\]
Restoring the view index gives
\(\mathbb E\|\Xi_s\|_F^2\le v_s^2\).
\end{proof}

\begin{lemma}[Gaussian distance-kernel bias--variance decomposition]
\label{lem:rbf-kernel-approx}
Let \(s\in\mathcal S_{\rm gauss}\), and suppose that the Gaussian
distance-kernel regime of Assumption~\ref{ass:admissible-noise}(2) holds.
Define
\[
    \Delta_s^{\rm gauss}
    :=
    K_s^{\rm obs}
    -
    \left(
        \alpha_sK_s^{\rm sig}+\chi_sI_N
    \right),
    \qquad
    \alpha_s:=\exp(-2\tau_s\sigma_s^2),
    \qquad
    \chi_s:=\frac{1-\alpha_s}{N},
\]
and
\[
    \eta_s^2
    :=
    \frac{
        \sigma_s^4
        \operatorname{tr}(\Psi_{s,m_s}^2)
    }{m_s^2}
    +
    \frac{
        \sigma_s^2
        \|\Psi_{s,m_s}\|_{\op}
        B_{s,N}
    }{m_s}.
\]
Then there is a decomposition
\[
    \Delta_s^{\rm gauss}
    =
    \Upsilon_s+\Xi_s,
    \qquad
    \mathbb E[\Xi_s]=0,
\]
such that
\[
   \|\Upsilon_s\|_F \le b_s,
   \qquad
   \mathbb E\|\Xi_s\|_F^2 \le v_s^2,
\]
where
\[
    b_s
    :=
    C_{s,b}
    \left(
        \tau_s^2\eta_s^2
        +
        \tau_s^3\eta_s^3
    \right),
    \qquad
    v_s^2
    :=
    C_{s,v}
    \left(
        \tau_s^2\eta_s^2
        +
        \tau_s^4\eta_s^4
        +
        \tau_s^6\eta_s^6
    \right),
\]
and \(C_{s,b}\) and \(C_{s,v}\) are constants independent of \(N\) and
\(m_s\).
\end{lemma}

\begin{proof}
We suppress the view index \(s\) and write
\[
    m=m_s,\qquad
    \sigma=\sigma_s,\qquad
    \tau=\tau_s,\qquad
    \Psi=\Psi_{s,m_s},\qquad
    B_N=B_{s,N},
\]
and
\[
    \alpha=e^{-2\tau\sigma^2},\qquad
    \chi=\frac{1-\alpha}{N},\qquad
    \eta=\eta_s.
\]
The signal vectors \(\{v_j\}_{j=1}^N\) are treated as fixed, and all
expectations below are taken with respect to the Gaussian noise only.

\paragraph{Step 1: Squared-distance perturbation.}
Set
\[
    \zeta_j
    :=
    \frac{\sigma\varepsilon_j}{\sqrt m},
    \qquad
    u_j=v_j+\zeta_j,
    \qquad
    \varepsilon_j
    \stackrel{\rm iid}{\sim}N(0,\Psi),
\]
and, for \(j\ne k\),
\[
    h_{jk}:=v_j-v_k,
    \qquad
    \rho_{jk}:=\|h_{jk}\|_2^2,
    \qquad
    a_{jk}:=\zeta_j-\zeta_k.
\]
Then \(a_{jk}\sim N(0,\Gamma)\) with \(\Gamma:=2\sigma^2\Psi/m\). Since
\(\operatorname{tr}(\Psi)/m=1\), we have
\(\operatorname{tr}(\Gamma)=2\sigma^2\), and consequently
\[
    \|u_j-u_k\|_2^2
    =
    \|h_{jk}+a_{jk}\|_2^2
    =
    \rho_{jk}+2\sigma^2+e_{jk},
\]
where
\[
    e_{jk}
    :=
    L_{jk}+Q_{jk},
    \qquad
    L_{jk}:=2h_{jk}^{\top}a_{jk},
    \qquad
    Q_{jk}
    :=
    \|a_{jk}\|_2^2-\operatorname{tr}(\Gamma).
\]
For notational convenience we also set \(h_{jj}:=0\) and \(e_{jj}:=0\), so
that the diagonal terms vanish in the sums over \(j,k\in[N]\) below.
Because \(a_{jk}\) is centered Gaussian, \(\mathbb E[e_{jk}]=0\). The cross
term \(\mathbb E[L_{jk}Q_{jk}]\) vanishes because \(L_{jk}\) is odd and
\(Q_{jk}\) is even in \(a_{jk}\). Therefore
\[
    \mathbb E[e_{jk}^2]
    =
    4h_{jk}^{\top}\Gamma h_{jk}
    +
    2\operatorname{tr}(\Gamma^2)
    =
    \frac{8\sigma^2}{m}
    h_{jk}^{\top}\Psi h_{jk}
    +
    \frac{8\sigma^4}{m^2}
    \operatorname{tr}(\Psi^2).
\]

\paragraph{Step 2: Higher moments of \(e_{jk}\).}
For \(q=1,2,3\), the elementary inequality
\(|a+b|^{2q}\le 2^{2q-1}(|a|^{2q}+|b|^{2q})\) gives
\[
    \mathbb E|e_{jk}|^{2q}
    \le
    C_q
    \left(
        \mathbb E|L_{jk}|^{2q}
        +
        \mathbb E|Q_{jk}|^{2q}
    \right).
\]
Since \(L_{jk}\) is centered Gaussian with variance
\(4h_{jk}^{\top}\Gamma h_{jk}\),
\[
    \mathbb E|L_{jk}|^{2q}
    =
    \mathbb E|G|^{2q}
    \left(
        4h_{jk}^{\top}\Gamma h_{jk}
    \right)^q
    \le
    C_q
    \left(
        h_{jk}^{\top}\Gamma h_{jk}
    \right)^q,
    \qquad G\sim N(0,1).
\]
For the quadratic term, diagonalize
\(\Gamma=U\operatorname{diag}(\lambda_1,\ldots,\lambda_m)U^\top\). If
\(g\sim N(0,I_m)\), then \(a_{jk}\stackrel{d}{=}\Gamma^{1/2}g\), and hence
\[
    Q_{jk}
    \stackrel{d}{=}
    \sum_{\ell=1}^m
    \lambda_\ell(g_\ell^2-1).
\]
Thus \(Q_{jk}\) belongs to the second-order homogeneous Gaussian chaos, and
Gaussian hypercontractivity \citep[Proposition~2.6]{nourdin2010invariance}
gives \(\|Q_{jk}\|_{L^p}\le(p-1)\|Q_{jk}\|_{L^2}\) for every \(p\ge2\).
Taking \(p=2q\) and using \(\mathbb E Q_{jk}^2=2\operatorname{tr}(\Gamma^2)\),
\[
    \mathbb E|Q_{jk}|^{2q}
    =
    \|Q_{jk}\|_{L^{2q}}^{2q}
    \le
    (2q-1)^{2q}
    \|Q_{jk}\|_{L^2}^{2q}
    \le
    C_q
    \left(
        \operatorname{tr}(\Gamma^2)
    \right)^q.
\]
Combining the linear and quadratic bounds yields, for \(q=1,2,3\),
\[
    \mathbb E|e_{jk}|^{2q}
    \le
    C_q
    \left[
        \left(
            h_{jk}^{\top}\Gamma h_{jk}
        \right)^q
        +
        \left(
            \operatorname{tr}(\Gamma^2)
        \right)^q
    \right]
    \le
    C_q
    \left[
        h_{jk}^{\top}\Gamma h_{jk}
        +
        \operatorname{tr}(\Gamma^2)
    \right]^q.
\]
We now use
\[
    h_{jk}^{\top}\Gamma h_{jk}
    \le
    \frac{2\sigma^2\|\Psi\|_{\op}}{m}
    \|h_{jk}\|_2^2,
    \qquad
    \operatorname{tr}(\Gamma^2)
    =
    \frac{4\sigma^4}{m^2}\operatorname{tr}(\Psi^2),
\]
together with the consequence of the signal-moment condition in
Assumption~\ref{ass:admissible-noise}(2),
\[
    \frac1{N^2}
    \sum_{j,k=1}^N
    \|h_{jk}\|_2^{2q}
    \le
    C_q B_N^q,
    \qquad q=1,2,3,
\]
to obtain
\begin{equation}
    \frac1{N^2}
    \sum_{j,k=1}^N
    \mathbb E|e_{jk}|^{2q}
    \le
    C_q\eta^{2q},
    \qquad q=1,2,3,
    \qquad
    \eta^2
    =
    \frac{
        \sigma^2\|\Psi\|_{\op}B_N
    }{m}
    +
    \frac{
        \sigma^4\operatorname{tr}(\Psi^2)
    }{m^2}.
    \label{eq:rbf-moment-bound}
\end{equation}

\paragraph{Step 3: Taylor expansion of the kernel residual.}
For \(j\ne k\), define
\[
    x_{jk}:=\rho_{jk}+2\sigma^2,
    \qquad
    y_{jk}:=\|u_j-u_k\|_2^2=x_{jk}+e_{jk},
\]
both of which lie in \([0,\infty)\), and let \(f(t):=e^{-\tau t}\) for
\(t\ge0\). Since \(\alpha K^{\rm sig}(j,k)=f(x_{jk})/N\), the off-diagonal
residual is
\[
    \Delta^{\rm gauss}(j,k)
    =
    \frac1N
    \left[
        f(y_{jk})-f(x_{jk})
    \right].
\]
On the diagonal, the observed and target kernels agree exactly:
\(K^{\rm obs}(j,j)=1/N\) and
\[
    \alpha K^{\rm sig}(j,j)+\chi
    =
    \frac{\alpha}{N}
    +
    \frac{1-\alpha}{N}
    =
    \frac1N,
\]
so \(\Delta^{\rm gauss}(j,j)=0\).

For \(j\ne k\), Taylor's theorem with integral
remainder~\citep{firey1960remainder} gives
\[
    f(y_{jk})-f(x_{jk})
    =
    f'(x_{jk})e_{jk}
    +
    \frac12f''(x_{jk})e_{jk}^2
    +
    \mathcal R_{jk},
\]
where
\[
    \mathcal R_{jk}
    =
    \frac{e_{jk}^3}{2}
    \int_0^1
    (1-r)^2
    f^{(3)}(x_{jk}+re_{jk})\,dr.
\]
The segment between \(x_{jk}\) and \(y_{jk}\) lies in \([0,\infty)\), where
\(|f^{(3)}(t)|=\tau^3e^{-\tau t}\le\tau^3\). Hence
\[
    |\mathcal R_{jk}|
    \le
    \frac{\tau^3}{6}|e_{jk}|^3,
    \qquad
    |\mathcal R_{jk}|^2
    \le
    \frac{\tau^6}{36}|e_{jk}|^6,
\]
and by~\eqref{eq:rbf-moment-bound} with \(q=3\),
\[
    \frac1{N^2}
    \sum_{j,k}
    \mathbb E|\mathcal R_{jk}|^2
    \le
    C\tau^6\eta^6.
\]
Define the random matrices
\[
    A^{(1)}(j,k)
    :=
    \mathbf 1_{\{j\ne k\}}
    \frac{f'(x_{jk})}{N}e_{jk},
    \qquad
    A^{(2)}(j,k)
    :=
    \mathbf 1_{\{j\ne k\}}
    \frac{f''(x_{jk})}{2N}e_{jk}^2,
\]
and
\[
    A^{(R)}(j,k)
    :=
    \mathbf 1_{\{j\ne k\}}
    \frac{\mathcal R_{jk}}{N},
\]
so that \(\Delta^{\rm gauss}=A^{(1)}+A^{(2)}+A^{(R)}\).

\paragraph{Step 4: Bias.}
Define the bias and the centered fluctuation by
\[
    \Upsilon
    :=
    \mathbb E[\Delta^{\rm gauss}],
    \qquad
    \Xi
    :=
    \Delta^{\rm gauss}-\Upsilon.
\]
Since \(\mathbb E[e_{jk}]=0\), we have \(\mathbb E[A^{(1)}]=0\), and hence
\(\Upsilon=\mathbb E[A^{(2)}]+\mathbb E[A^{(R)}]\). Because
\(|f''(x)|\le\tau^2\) for \(x\ge0\), Jensen's inequality gives
\[
    \|\mathbb E A^{(2)}\|_F^2
    =
    \sum_{j\ne k}
    \left|
        \frac{f''(x_{jk})}{2N}
        \mathbb E[e_{jk}^2]
    \right|^2
    \le
    \frac{\tau^4}{4N^2}
    \sum_{j\ne k}
    \mathbb E[e_{jk}^4]
    \le
    C\tau^4\eta^4.
\]
Similarly,
\[
    \|\mathbb E A^{(R)}\|_F^2
    \le
    \frac1{N^2}
    \sum_{j\ne k}
    \mathbb E|\mathcal R_{jk}|^2
    \le
    C\tau^6\eta^6.
\]
Therefore
\[
    \|\Upsilon\|_F
    \le
    \|\mathbb E A^{(2)}\|_F
    +
    \|\mathbb E A^{(R)}\|_F
    \le
    C
    \left(
        \tau^2\eta^2
        +
        \tau^3\eta^3
    \right).
\]

\paragraph{Step 5: Fluctuation.}
Write
\[
    \Xi
    =
    A^{(1)}
    +
    \left(
        A^{(2)}-\mathbb E A^{(2)}
    \right)
    +
    \left(
        A^{(R)}-\mathbb E A^{(R)}
    \right).
\]
For the first-order term,
\[
    \mathbb E\|A^{(1)}\|_F^2
    =
    \frac1{N^2}
    \sum_{j\ne k}
    f'(x_{jk})^2
    \mathbb E[e_{jk}^2]
    \le
    C\tau^2\eta^2.
\]
For the centered second-order and remainder terms,
\[
    \mathbb E
    \left\|
        A^{(2)}-\mathbb E A^{(2)}
    \right\|_F^2
    \le
    \mathbb E\|A^{(2)}\|_F^2
    \le
    C\tau^4\eta^4,
    \qquad
    \mathbb E
    \left\|
        A^{(R)}-\mathbb E A^{(R)}
    \right\|_F^2
    \le
    \mathbb E\|A^{(R)}\|_F^2
    \le
    C\tau^6\eta^6.
\]
Using \(\|M_1+M_2+M_3\|_F^2\le3(\|M_1\|_F^2+\|M_2\|_F^2+\|M_3\|_F^2)\), we
conclude that
\[
    \mathbb E\|\Xi\|_F^2
    \le
    C
    \left(
        \tau^2\eta^2
        +
        \tau^4\eta^4
        +
        \tau^6\eta^6
    \right).
\]
Restoring the view index \(s\), the bounds
\(\|\Upsilon_s\|_F\le b_s\) and \(\mathbb E\|\Xi_s\|_F^2\le v_s^2\) hold
with \(b_s\) and \(v_s^2\) as defined in the statement.
\end{proof}
\begin{lemma}[Mixed target eigengap]
\label{lem:mixed-target-gap}
For \(w\in\Delta_S\), define
\[
    K_{\rm tar}(w)
    :=
    \sum_{s=1}^S
    w_s\alpha_sK_s^{\rm sig}
    +
    \left(
        \sum_{s=1}^S w_s\chi_s
    \right)I_N .
\]
Under Assumptions~\ref{ass:ek-shared-subspace} and~\ref{ass:ek-eigengap},
the top-\(d\) eigenspace of \(K_{\rm tar}(w)\) is
\(\operatorname{span}(Z_{\rm sig})\), and its eigengap is at least
\[
    \gamma(w)
    :=
    \lambda_{\min}
    \left(
        \sum_{s=1}^S w_s\alpha_s\Lambda_s
    \right)
    -
    \left\|
        \sum_{s=1}^S w_s\alpha_sR_s^{\rm sig}
    \right\|_{\op}.
\]
Moreover, with the per-view gaps
\(
    \gamma_s
    :=
    \lambda_{\min}(\Lambda_s)-\|R_s^{\rm sig}\|_{\op}
\),
\[
    \gamma(w)
    \ge
    \sum_{s=1}^S w_s\alpha_s\gamma_s.
\]
\end{lemma}

\begin{proof}
By the shared-subspace decomposition in Assumption~\ref{ass:ek-shared-subspace},
\[
    K_s^{\rm sig}
    =
    Z_{\rm sig}\Lambda_sZ_{\rm sig}^{\top}
    +
    R_s^{\rm sig},
    \qquad
    R_s^{\rm sig}Z_{\rm sig}=0,
    \qquad
    Z_{\rm sig}^{\top}R_s^{\rm sig}=0.
\]
Therefore
\[
    K_{\rm tar}(w)
    =
    Z_{\rm sig}
    \left(
        \sum_{s=1}^S
        w_s\alpha_s\Lambda_s
    \right)
    Z_{\rm sig}^{\top}
    +
    \sum_{s=1}^S
    w_s\alpha_sR_s^{\rm sig}
    +
    \left(
        \sum_{s=1}^S
        w_s\chi_s
    \right)I_N.
\]
The last term is a scalar multiple of \(I_N\); it shifts all eigenvalues
equally and changes neither the eigenspaces nor the eigengaps, so we may
ignore it. The remaining matrix is block diagonal with respect to the
decomposition
\[
    \mathbb R^N
    =
    \operatorname{span}(Z_{\rm sig})
    \oplus
    \operatorname{span}(Z_{\rm sig})^\perp.
\]
On \(\operatorname{span}(Z_{\rm sig})\) it acts as
\(\sum_{s}w_s\alpha_s\Lambda_s\) in the orthonormal basis \(Z_{\rm sig}\), so
its \(d\) eigenvalues there are at least
\(\lambda_{\min}(\sum_{s}w_s\alpha_s\Lambda_s)\). On
\(\operatorname{span}(Z_{\rm sig})^\perp\) it acts as
\(\sum_{s}w_s\alpha_sR_s^{\rm sig}\), so its remaining eigenvalues are at most
\(\|\sum_{s}w_s\alpha_sR_s^{\rm sig}\|_{\op}\). Hence
\[
    \lambda_d\{K_{\rm tar}(w)\}-\lambda_{d+1}\{K_{\rm tar}(w)\}
    \ge
    \gamma(w).
\]
Whenever \(\gamma(w)>0\), which holds by Assumption~\ref{ass:ek-eigengap} and
the bound below, the \(d\) largest eigenvalues are exactly those on
\(\operatorname{span}(Z_{\rm sig})\), so the top-\(d\) eigenspace is
\(\operatorname{span}(Z_{\rm sig})\).

Finally, since \(\lambda_{\min}\) is superadditive on symmetric matrices and
the operator norm is subadditive,
\[
    \gamma(w)
    \ge
    \sum_{s=1}^S
    w_s\alpha_s\lambda_{\min}(\Lambda_s)
    -
    \sum_{s=1}^S
    w_s\alpha_s\|R_s^{\rm sig}\|_{\op}
    =
    \sum_{s=1}^S
    w_s\alpha_s\gamma_s.
\]
\end{proof}

\begin{proof}[Proof of Theorem~\ref{thm:main-pretrain}]
Let \(K_{\rm tar}(w)\) be the target matrix of
Lemma~\ref{lem:mixed-target-gap}. By that lemma, the top-\(d\) eigenspace of
\(K_{\rm tar}(w)\) is \(\operatorname{span}(Z_{\rm sig})\), with eigengap at
least \(\gamma(w)\).

For each view, let $\Upsilon_s$, $\Xi_s$, $b_s:=b_{s,N,m_s}$, $d_s:=d_{s,N,m_s}$, and $v_s:=v_{s,N,m_s}$ denote the
bias matrix, the centered fluctuation matrix, and the corresponding error
scales defined in Appendix~\ref{app:kernel-regimes} and
Lemmas~\ref{lem:ip-kernel-approx}, \ref{lem:linear-kernel-approx},
and~\ref{lem:rbf-kernel-approx}, according to the kernel type. Then
\[
    K_s^{\rm obs}
    -
    \alpha_sK_s^{\rm sig}
    -
    \chi_sI_N
    =
    \Upsilon_s+\Xi_s,
    \qquad
    \E\,\Xi_s=0,
\]
and
\[
    \|\Upsilon_s\|_F
    \le b_s+\mathbf1_{\{s\in \mathcal S_{\rm ip}\}}d_s,
    \qquad
    \E\|\Xi_s\|_F^2
    \le
    v_s^2.
\]
Since the centered matrices are independent across views and have mean zero,
\(
    \E\langle\Xi_s,\Xi_t\rangle_F
    =
    \langle\E\,\Xi_s,\E\,\Xi_t\rangle_F
    =
    0
\)
for \(s\ne t\). Consequently,
\[
    \E
    \left\|
        \sum_{s=1}^S
        w_s\Xi_s
    \right\|_F^2
    =
    \sum_{s=1}^S
    w_s^2
    \E\|\Xi_s\|_F^2
    \le
    \sum_{s=1}^S
    w_s^2v_s^2 .
\]
Markov's inequality applied to the squared Frobenius norm gives
\[
    \mathbb P
    \left\{
        \left\|
            \sum_{s=1}^S
            w_s\Xi_s
        \right\|_F
        >
        \left(
            \frac1\delta
            \sum_{s=1}^S
            w_s^2v_s^2
        \right)^{1/2}
    \right\}
    \le
    \delta.
\]
Hence, by the triangle inequality, with probability at least \(1-\delta\),
\[
    \|K^{\rm obs}(w)-K_{\rm tar}(w)\|_F
    \le
    \sum_{s=1}^S w_sb_s+\sum_{s\in \mathcal S_{\rm ip}}w_sd_s
    +
    \left(
        \frac1\delta
        \sum_{s=1}^S
        w_s^2v_s^2
    \right)^{1/2}.
\]
The Frobenius-norm form of the Davis--Kahan \(\sin\Theta\)
theorem~\citep{yu2015useful} gives
\[
    \|\sin\Theta(\widehat Z,Z_{\rm sig})\|_F
    \le
    \frac{C}{\gamma(w)}
    \|K^{\rm obs}(w)-K_{\rm tar}(w)\|_F,
\]
so, with probability at least \(1-\delta\),
\[
    \|\sin\Theta(\widehat Z,Z_{\rm sig})\|_F
    \le
    \frac{C}{\gamma(w)}
    \left[
        \sum_{s=1}^S
        w_s b_s
        +
        \sum_{s\in \mathcal S_{\rm ip}}w_sd_s
        +
        \left(
            \frac1\delta
            \sum_{s=1}^S
            w_s^2v_s^2
        \right)^{1/2}
    \right].
\]
Squaring and using \((a+b)^2\le 2(a^2+b^2)\), we obtain, after enlarging
\(C\),
\[
    \|\sin\Theta(\widehat Z,Z_{\rm sig})\|_F^2
    \le
    \frac{C}{\gamma(w)^2}
    \left[
        \left(
            \sum_{s=1}^S w_sb_s+\sum_{s\in \mathcal S_{\rm ip}}w_sd_s
        \right)^2
        +
        \frac1{\delta}
        \sum_{s=1}^S w_s^2v_s^2
    \right],
\]
which is~\eqref{eq:main-pretrain-bound}.
\end{proof}
\subsection{Proof of Theorem~\ref{thm:main-end2end} (end-to-end risk bound)}
\label{app:proof-end2end}

We now prove the end-to-end bound. The argument conditions on the unlabeled pretraining sample, so the empirical KPCA coordinate matrix $\widehat Z$ is fixed throughout the supervised analysis. The proof has three components: a fixed-embedding oracle inequality, a transfer bound from the empirical KPCA coordinates to the true latent coordinates, and a ReLU approximation bound for the smooth relation surfaces. Write
\[
    \mathcal D_\star:=[-R_\star,R_\star]^d\times[-R_\star,R_\star]^d.
\]
By Assumption~\ref{ass:smooth-kg-app}, every latent coordinate satisfies $\|z_j^\star\|_2\le R_\star-1$, so $\mathcal D_\star$ contains all pairs of latent coordinates with a margin of $1$. The transformed KPCA coordinates introduced below need not lie in $\mathcal D_\star$; the proof of Lemma~\ref{lem:oracle-transfer-app-expanded} handles the entities whose coordinates are poorly recovered separately.
Throughout this subsection, $\cF$ denotes the class of relation-wise head collections $f=(f_1,\ldots,f_K)$, where each $f_r$ has the form~\eqref{eq:relation-wise-head} with depth at most $D$, at most $W$ parameters, and clipping level $B$; for such $f$ we write $f(\widehat Z;x):=f_r(\widehat z_h,\widehat z_t)$ for $x=(h,r,t)$, so that $\{x\mapsto f(\widehat Z;x):f\in\cF\}=\cF_{\widehat Z}(D,W,B)$, conditionally on the pretraining sample.

\begin{lemma}[Fixed-embedding oracle inequality]
\label{lem:fixed-oracle-app}
Condition on the pretraining sample and define the frozen class
\[
    \cF_{\widehat Z}:=\{x\mapsto f(\widehat Z;x):f\in\cF\}=\cF_{\widehat Z}(D,W,B).
\]
Let $p$ be any real number with
$\operatorname{Pdim}(\cF_{\widehat Z})\le p\le n$. Under Assumption~\ref{ass:sup-sampling-app}, with conditional probability at least
$1-10(p/en)^{p}$,
\begin{equation}
\label{eq:fixed-oracle-app-expanded}
    \|\widehat f(\widehat Z;\cdot)-f^\star\|_P^2
    \le
    27\inf_{f\in\cF}\|f(\widehat Z;\cdot)-f^\star\|_P^2
    +C_{\rm or}(\sigma_y^2+B^2)\frac{p}{n}\log\frac{en}{p}
    +6\delta_{\rm opt},
\end{equation}
where $C_{\rm or}$ is an absolute constant.
Furthermore, for relation-wise ReLU heads with common depth $D$ and parameter budget $W$, there is an absolute constant $C_{\rm pd}\ge1$ such that, for every $\widehat Z$,
\[
    \operatorname{Pdim}(\cF_{\widehat Z})\le C_{\rm pd}\,KDW\log W.
\]
\end{lemma}

\begin{proof}
Conditionally on the pretraining sample, $\widehat Z$ is deterministic, and, because the labeled sample is independent of the pretraining sample (Assumption~\ref{ass:pretrain-supervised-independence}), the only randomness in the supervised stage comes from the labeled sample. The standard bounded least-squares oracle inequality for a fixed function class therefore applies to $\cF_{\widehat Z}$: Theorem~1 of \citet{liu2024representation}, in the present notation, gives~\eqref{eq:fixed-oracle-app-expanded} for any $\delta_{\rm opt}$-approximate empirical risk minimizer.
That theorem is stated with $p=\operatorname{Pdim}(\cF_{\widehat Z})$, but its proof uses the pseudo-dimension only through the uniform covering bound of \citet[Theorem~12.2]{anthony1999neural},
\[
    \mathcal N_\infty(\epsilon,\cF_{\widehat Z},n)
    \le
    \sum_{i=0}^{\operatorname{Pdim}(\cF_{\widehat Z})}\binom ni\Big(\frac {2B}\epsilon\Big)^i,
\]
where $2B$ is the length of the range $[-B,B]$ of the clipped heads, applied with $\epsilon\le 2B$; every other occurrence of $p$ is a free parameter of the argument. For $x:=2B/\epsilon\ge1$ and any real $p$ with $\operatorname{Pdim}(\cF_{\widehat Z})\le p\le n$, each index satisfies $i\le p$, so $x^i\le x^p$ and $(p/n)^{i}\ge(p/n)^{p}$; hence
\[
    \sum_{i=0}^{\operatorname{Pdim}(\cF_{\widehat Z})}\binom ni x^i
    \le
    x^p\Big(\frac np\Big)^p\sum_{i=0}^{n}\binom ni\Big(\frac pn\Big)^i
    =
    x^p\Big(\frac np\Big)^p\Big(1+\frac pn\Big)^n
    \le
    \Big(\frac{2enB}{\epsilon p}\Big)^p .
\]
This is the covering bound used in that proof, now with $p$ in place of the pseudo-dimension, so the conclusion holds for every such $p$. The pseudo-dimension bound follows from the standard estimate for ReLU networks, as in Lemma~3 of \citet{liu2024representation}; with $K$ heads of common architecture $(D,W)$, the contributions add over relations and give $\operatorname{Pdim}(\cF_{\widehat Z})\le C_{\rm pd}KDW\log W$.
\end{proof}

\begin{lemma}[Lipschitz property of the smooth oracle surfaces]
\label{lem:lipschitz-app}
Under Assumption~\ref{ass:smooth-kg-app}, since $\beta\ge1$, each relation surface $g_r^\star$ is Lipschitz on $\mathcal D_\star$. More precisely, there exists
$L_\star\le C(d)S_\beta$ such that, for all $(a,b),(a',b')\in\mathcal D_\star$,
\[
    |g_r^\star(a,b)-g_r^\star(a',b')|
    \le
    L_\star\big(\|a-a'\|_2+\|b-b'\|_2\big).
\]
\end{lemma}

\begin{proof}
The bound $\|g_r^\star\|_{W^{\beta,\infty}}\le S_\beta$ with $\beta\ge1$ implies that all first-order weak derivatives are essentially bounded by $S_\beta$, so $\sup_z\|\nabla g_r^\star(z)\|_2\le C(d)S_\beta$. Since $\mathcal D_\star$ is convex, the mean-value inequality gives
\[
    |g_r^\star(u)-g_r^\star(v)|
    \le
    \sup_{z}\|\nabla g_r^\star(z)\|_2\,\|u-v\|_2
    \le C(d)S_\beta\|u-v\|_2.
\]
Writing $u=(a,b)$ and $v=(a',b')$ and using
$\|(a-a',b-b')\|_2\le\|a-a'\|_2+\|b-b'\|_2$ gives the claim.
\end{proof}

\begin{lemma}[KPCA coordinate alignment]
\label{lem:kpca-align-app}
Let
\[
    Q\in\arg\min_{U\in O(d)}\|\widehat ZU-Z_{\rm sig}\|_F,
    \qquad
    \widetilde Z:=\sqrt N\,\widehat ZQA_\star^{-1},
\]
and let $\widetilde z_j$ denote the $j$th row of $\widetilde Z$.
Under Assumption~\ref{ass:alignment-app}, the entity-averaged coordinate error satisfies
\[
    \bar\Delta^2
    :=
    \frac1N\sum_{j=1}^N\|\widetilde z_j-z_j^\star\|_2^2
    \le
    2\kappa_A^2
    \|\sin\Theta(\widehat Z,Z_{\rm sig})\|_F^2 .
\]
\end{lemma}

\begin{proof}
By the standard Procrustes bound for principal angles,
\[
    \min_{U\in O(d)}\|\widehat ZU-Z_{\rm sig}\|_F
    \le
    \sqrt2\,\|\sin\Theta(\widehat Z,Z_{\rm sig})\|_F.
\]
Since $\sqrt N Z_{\rm sig}=Z^\star A_\star$,
\[
    \widetilde Z-Z^\star
    =
    \sqrt N\,\widehat ZQA_\star^{-1}-\sqrt N Z_{\rm sig}A_\star^{-1}
    =
    \sqrt N\,(\widehat ZQ-Z_{\rm sig})A_\star^{-1}.
\]
Therefore
\[
    \bar\Delta^2
    =
    \frac1N\|\widetilde Z-Z^\star\|_F^2
    \le
    \|\widehat ZQ-Z_{\rm sig}\|_F^2\,\|A_\star^{-1}\|_{\op}^2
    \le
    2\kappa_A^2\|\sin\Theta(\widehat Z,Z_{\rm sig})\|_F^2.
\]
\end{proof}

\begin{lemma}[Transfer of the oracle bias to empirical KPCA coordinates]
\label{lem:oracle-transfer-app-expanded}
Suppose Assumptions~\ref{ass:sup-sampling-app}--\ref{ass:alignment-app} hold, and that there exists a relation-wise head collection $f^{\rm app}\in\cF$, each of whose heads has at most $W/(2d)$ parameters, such that
\[
    \sup_{r\in[K]}
    \sup_{(a,b)\in\mathcal D_\star}|f_r^{\rm app}(a,b)-g_r^\star(a,b)|\le \varepsilon_{\rm app}.
\]
Then, with
\(
    \bar L:=\max\{L_\star,2B\}
\)
and $\bar\Delta$ as in Lemma~\ref{lem:kpca-align-app},
\begin{equation}
\label{eq:transfer-oracle-app-expanded}
    \inf_{f\in\cF}\|f(\widehat Z;\cdot)-f^\star\|_P^2
    \le
    2\varepsilon_{\rm app}^2+8\bar L^2\bar\Delta^2,
\end{equation}
and consequently
\begin{equation}
\label{eq:transfer-oracle-sin-app-expanded}
    \inf_{f\in\cF}\|f(\widehat Z;\cdot)-f^\star\|_P^2
    \le
    2\varepsilon_{\rm app}^2
    +16\bar L^2\kappa_A^2
    \|\sin\Theta(\widehat Z,Z_{\rm sig})\|_F^2.
\end{equation}
\end{lemma}

\begin{proof}
Let $Q$ and $\widetilde Z$ be as in Lemma~\ref{lem:kpca-align-app}, and set
\[
    A:=\sqrt N\,QA_\star^{-1},
    \qquad\text{so that}\qquad
    \widetilde z_j=A^\top\widehat z_j,\quad j\in[N].
\]
For a relation-wise head $f$, define the linear reparameterization
\[
    f_r^{[A]}(a,b):=f_r(A^\top a,A^\top b).
\]
For ReLU networks, precomposition with the fixed invertible map $(a,b)\mapsto(A^\top a,A^\top b)$ is absorbed into the first affine layer: the first-layer weight matrix is multiplied on the right by the $2d\times2d$ matrix $\operatorname{blockdiag}(A^\top,A^\top)$, which leaves the depth and all other layers unchanged and gives each first-layer row at most $2d$ nonzero entries. Hence each reparameterized head has at most $2d$ times as many parameters as the original head, that is, at most $W$, and the output clipping is unaffected. Therefore
\[
    \widetilde f^{\rm app}:=(f^{\rm app})^{[A]}\in\cF.
\]
Because the heads in $\cF$ are clipped, $\widetilde f^{\rm app}$ takes values in $[-B,B]$.

Write $\Delta_j:=\|\widetilde z_j-z_j^\star\|_2$ for $j\in[N]$, and fix a triple $x=(h,r,t)$. By construction,
\[
    \widetilde f^{\rm app}(\widehat Z;x)
    =\widetilde f_r^{\rm app}(\widehat z_h,\widehat z_t)
    =f_r^{\rm app}(\widetilde z_h,\widetilde z_t),
    \qquad
    f^\star(x)=g_r^\star(z_h^\star,z_t^\star).
\]
We claim that
\begin{equation}
\label{eq:transfer-pointwise}
    |\widetilde f^{\rm app}(\widehat Z;x)-f^\star(x)|
    \le
    \varepsilon_{\rm app}+\bar L(\Delta_h+\Delta_t).
\end{equation}
\emph{Case 1: $\Delta_h\le1$ and $\Delta_t\le1$.}
By Assumption~\ref{ass:smooth-kg-app},
$\|\widetilde z_h\|_2\le\|z_h^\star\|_2+\Delta_h\le R_\star$, and likewise for $t$;
hence every coordinate of $\widetilde z_h$ and $\widetilde z_t$ lies in
$[-R_\star,R_\star]$, and both $(\widetilde z_h,\widetilde z_t)$ and
$(z_h^\star,z_t^\star)$ lie in $\mathcal D_\star$. By the approximation
hypothesis and Lemma~\ref{lem:lipschitz-app},
\[
\begin{aligned}
    |\widetilde f^{\rm app}(\widehat Z;x)-f^\star(x)|
    &\le
    |f_r^{\rm app}(\widetilde z_h,\widetilde z_t)-g_r^\star(\widetilde z_h,\widetilde z_t)|
    +
    |g_r^\star(\widetilde z_h,\widetilde z_t)-g_r^\star(z_h^\star,z_t^\star)|  \\
    &\le
    \varepsilon_{\rm app}+L_\star(\Delta_h+\Delta_t).
\end{aligned}
\]
\emph{Case 2: $\Delta_h>1$ or $\Delta_t>1$.}
Both $\widetilde f^{\rm app}(\widehat Z;x)$ and $f^\star(x)$ lie in $[-B,B]$, by
the clipping and by Assumption~\ref{ass:sup-sampling-app}, respectively, so
\[
    |\widetilde f^{\rm app}(\widehat Z;x)-f^\star(x)|
    \le
    2B
    <
    2B(\Delta_h+\Delta_t).
\]
In both cases~\eqref{eq:transfer-pointwise} holds.

Squaring~\eqref{eq:transfer-pointwise} and using $(x+y)^2\le2x^2+2y^2$ twice gives
\[
    |\widetilde f^{\rm app}(\widehat Z;x)-f^\star(x)|^2
    \le
    2\varepsilon_{\rm app}^2+4\bar L^2(\Delta_h^2+\Delta_t^2).
\]
By Assumption~\ref{ass:sup-sampling-app}, the head and the tail of $X\sim P$
are each uniformly distributed on $[N]$, so
\(
    \E_{X\sim P}\Delta_h^2
    =
    \E_{X\sim P}\Delta_t^2
    =
    N^{-1}\sum_{j=1}^N\Delta_j^2
    =
    \bar\Delta^2
\).
Taking the expectation over $X\sim P$ therefore gives
\[
    \|\widetilde f^{\rm app}(\widehat Z;\cdot)-f^\star\|_P^2
    \le
    2\varepsilon_{\rm app}^2+8\bar L^2\bar\Delta^2.
\]
Since $\widetilde f^{\rm app}\in\cF$, this proves~\eqref{eq:transfer-oracle-app-expanded}. Substituting Lemma~\ref{lem:kpca-align-app} proves~\eqref{eq:transfer-oracle-sin-app-expanded}.
\end{proof}

\begin{lemma}[Yarotsky approximation on the relation surfaces]
\label{lem:yarotsky-app}
There are constants $C_{\rm app}$ and $W_0$ depending only on $(d,\beta,R_\star,S_\beta,B)$ such that, for all $W\ge W_0$ and $D\gtrsim\log W$, there exists $f^{\rm app}\in\cF$, each of whose heads has at most $W/(2d)$ parameters, satisfying
\begin{equation}
\label{eq:yarotsky-app-expanded}
    \sup_{r\in[K]}\sup_{(a,b)\in\mathcal D_\star}
    |f_r^{\rm app}(a,b)-g_r^\star(a,b)|
    \le
    C_{\rm app}\Big(\frac{W}{\log W}\Big)^{-\beta/(2d)}.
\end{equation}
\end{lemma}

\begin{proof}
Each $g_r^\star$ is defined on a $2d$-dimensional cube and satisfies a uniform
$W^{\beta,\infty}$ bound. After an affine rescaling of the input cube to $[0,1]^{2d}$ and a harmless output normalization, the ReLU approximation theorem of \citet{yarotsky2017error} gives, for every $\varepsilon\in(0,1)$, a ReLU network of depth
$O(\log(1/\varepsilon))$ with
$O(\varepsilon^{-2d/\beta}\log(1/\varepsilon))$ weights whose uniform error is at most $C\varepsilon$. Applying this construction separately to the $K$ relation surfaces under the common architecture budget gives a relation-wise collection in $\cF$. Since the heads have bounded outputs, the clipping map
$u\mapsto -B+(u+B)_+-(u-B)_+$ can be appended at constant cost.
Choosing
\[
    \varepsilon_W\asymp \Big(\frac{W}{\log W}\Big)^{-\beta/(2d)}
\]
keeps the parameter count of each head at most $W/(2d)$ and the depth $O(\log W)$; the factor $2d$ only changes $C_{\rm app}$ and $W_0$, since $d$ is fixed. This proves~\eqref{eq:yarotsky-app-expanded}.
\end{proof}

\paragraph{Completion of the proof.}
\emph{Step 1: conditional bound.}
Condition on the pretraining sample. By Lemma~\ref{lem:yarotsky-app}, the hypothesis of Lemma~\ref{lem:oracle-transfer-app-expanded} holds with
\[
    \varepsilon_{\rm app}=C_{\rm app}\Big(\frac{W}{\log W}\Big)^{-\beta/(2d)},
\]
so~\eqref{eq:transfer-oracle-sin-app-expanded} gives
\[
    \inf_{f\in\cF}\|f(\widehat Z;\cdot)-f^\star\|_P^2
    \le
    2C_{\rm app}^2\Big(\frac{W}{\log W}\Big)^{-\beta/d}
    +16\bar L^2\kappa_A^2
    \|\sin\Theta(\widehat Z,Z_{\rm sig})\|_F^2.
\]
Since $\bar L^2\le L_\star^2+4B^2\le4(L_\star^2+B^2)$, the last term is at most $64(L_\star^2+B^2)\kappa_A^2\|\sin\Theta(\widehat Z,Z_{\rm sig})\|_F^2$.
Substituting this bound into Lemma~\ref{lem:fixed-oracle-app} shows that, with conditional probability at least $1-10(p/en)^{p}$,
\begin{equation}
\label{eq:end2end-conditional-app}
\begin{aligned}
    \|\widehat f(\widehat Z;\cdot)-f^\star\|_P^2
    \le{}&
    54C_{\rm app}^2\Big(\frac{W}{\log W}\Big)^{-\beta/d}
    +1728(L_\star^2+B^2)\kappa_A^2
    \|\sin\Theta(\widehat Z,Z_{\rm sig})\|_F^2\\
    &+C_{\rm or}(\sigma_y^2+B^2)\frac{p}{n}\log\frac{en}{p}
    +6\delta_{\rm opt}.
\end{aligned}
\end{equation}
We apply Lemma~\ref{lem:fixed-oracle-app} with the deterministic choice
\[
    p:=C_{\rm pd}\,KDW\log W,
\]
which satisfies $\operatorname{Pdim}(\cF_{\widehat Z})\le p\le n$ for every realization of $\widehat Z$ whenever $C_{\rm pd}KDW\log W\le n$. Since $C_{\rm pd}\ge1$,
\[
    \frac{p}{n}\log\frac{en}{p}
    \le
    C_{\rm pd}\,\frac{KDW\log W}{n}\log\frac{en}{KDW\log W}.
\]

\emph{Step 2: combining the two stages.}
Let
\[
    \mathcal D_{\rm pt}:=\{E_s:s\in[S]\}
\]
denote the unlabeled pretraining data, equivalently
$\{u_{s,j}:s\in[S],\,j\in[N]\}$. Define the pretraining event
\[
    \Omega_{\rm pt}
    :=
    \left\{
    \|\sin\Theta(\widehat Z,Z_{\rm sig})\|_{F}
    \le a_{\rm pt}
    \right\},
\]
and let $\Omega_{\rm sup}(\widehat Z)$ denote the event on which~\eqref{eq:end2end-conditional-app} holds, conditionally on $\mathcal D_{\rm pt}$.
Suppose that
\[
    \mathbb P(\Omega_{\rm pt})\ge 1-\delta_{\rm pt}
\]
and that, for every realization of $\mathcal D_{\rm pt}$,
\[
    \mathbb P\!\left(
        \Omega_{\rm sup}(\widehat Z)
        \mid
        \mathcal D_{\rm pt}
    \right)
    \ge
    1-\delta_{\rm sup}.
\]
Then
\[
\begin{aligned}
    \mathbb P\!\left(
        \Omega_{\rm pt}
        \cap
        \Omega_{\rm sup}(\widehat Z)
    \right)
    &=
    \E\!\left[
        \mathbf 1_{\Omega_{\rm pt}}\,
        \mathbb P\!\left(
            \Omega_{\rm sup}(\widehat Z)
            \mid
            \mathcal D_{\rm pt}
        \right)
    \right]
    \ge
    (1-\delta_{\rm sup})\,\mathbb P(\Omega_{\rm pt})  \\
    &\ge
    (1-\delta_{\rm sup})(1-\delta_{\rm pt})
    \ge
    1-\delta_{\rm sup}-\delta_{\rm pt}.
\end{aligned}
\]

\emph{Step 3: conclusion.}
Let $C_{\rm pt}$ denote the constant $C$ of Theorem~\ref{thm:main-pretrain}. By that theorem, we may take
\[
    a_{\rm pt}^2=C_{\rm pt}\,\mathcal E_{\rm pt}(w,\delta),
    \qquad
    \delta_{\rm pt}=\delta,
\]
and, by Lemma~\ref{lem:fixed-oracle-app} with $p=C_{\rm pd}KDW\log W$ as in Step~1, $\delta_{\rm sup}=10(p/en)^p$, which does not depend on $\mathcal D_{\rm pt}$. On $\Omega_{\rm pt}\cap\Omega_{\rm sup}(\widehat Z)$, inserting $\|\sin\Theta(\widehat Z,Z_{\rm sig})\|_F^2\le C_{\rm pt}\mathcal E_{\rm pt}(w,\delta)$ and the pseudo-dimension bound of Step~1 into~\eqref{eq:end2end-conditional-app} yields~\eqref{eq:main-end2end-bound} with probability at least $1-10(p/en)^p-\delta$ and constants
\[
    C_1=54\,C_{\rm app}^2,
    \qquad
    C_2=C_{\rm or}C_{\rm pd},
    \qquad
    C_3=1728\,C_{\rm pt},
    \qquad
    C_4=6.
\]
Thus $C_1$ depends only on $(d,\beta,R_\star,S_\beta,B)$, $C_2$ is absolute, $C_3$ depends only on the view-level noise constants of Theorem~\ref{thm:main-pretrain}, and $C_4$ is absolute. This completes the proof of Theorem~\ref{thm:main-end2end}.

\section{Simulation Details}
\label{app:simulations}
\subsection{Pretraining recovery and view weighting}
For pretraining, we use the linear kernel model:
\[u_{s,j}=v_{s,j}+\frac{1}{\sqrt{m_s}}\varepsilon_{s,j}, \qquad j=1,\ldots,N, \] where \(u_{s,j}\in\R^{m_s}\) is the observed entity vector for entity \(j\) in view \(s\), \(v_{s,j}\) is the corresponding noiseless vector, and \(\varepsilon_{s,j}\) is Gaussian noise across entities. We use \[ N=60,\qquad d=4,\qquad S=4, \] and a shared latent subspace \(Z_{\rm sig}\in\R^{N\times d}\) satisfying \(Z_{\rm sig}^\top Z_{\rm sig}=I_d\), the shared signal eigenspace of the theory. For each view, \[ V_s=Z_{\rm sig} T_s^\top, \qquad T_s^\top T_s=N\Lambda_s, \] so that the noiseless linear kernel satisfies \( \frac{1}{N}V_sV_s^\top = Z_{\rm sig}\Lambda_s Z_{\rm sig}^\top \).
The view-specific signal spectra are \[ \operatorname{diag}(\Lambda_s)\in \{(10,7,4.5,2),\, (7,5,3,1.5),\, (4.8,3.5,2,1),\, (8.5,6,3.5,1.8)\}. \]
Each 4-tuple gives the \(d=4\) nonzero signal eigenvalues for one view; there
are four tuples because~\(S=4\).

We compare two noise regimes.
(i) In the isotropic regime, \[ \varepsilon_{s,j} \sim N(0,\sigma^2 I), \qquad \sigma=0.25 . \] (ii) In the anisotropic regime, \[ \varepsilon_{s,j}\sim N(0,\Sigma_s), \qquad \Sigma_s=\sigma_s^2 I+Q_sD_sQ_s^\top , \] where the anisotropic low-rank component is fixed for each view and is independent of \(Z_{\rm sig}\). The view-specific base noise levels are \[ (\sigma_s)=(0.15,0.20,0.25,0.30), \] with low-rank scales \( (0.8,1.2,2.0,2.8) \) and ranks \(6,6,8,8\), respectively.

For a given weight vector \(w\), we form the aggregated empirical linear kernel
\[
    \widehat K(w)
    =
    \sum_{s=1}^S w_s\,\frac{1}{N}E_sE_s^\top .
\] The reported subspace recovery error is the operator-norm principal-angle error \[ \|\sin\Theta(\widehat Z,Z_{\rm sig})\|_{\op}, \] where \(\widehat Z\) denotes the leading \(d\)-dimensional eigenspace of the aggregated empirical kernel \(\widehat K(w)\).  All reported means and standard deviations are computed
over independent noise realizations, with \(Z_{\rm sig}\), the noiseless view
matrices \(V_s\), and the noise covariance matrices fixed throughout each
experiment.

Experiment~1 varies the common feature dimension \[ m_s=m\in\{50,100,200,400,800\} \] under uniform view weights. For each value of \(m\) and each noise
regime, we compute
the resulting subspace error over 20 runs. The isotropic-noise curve follows the predicted \(m^{-1/2}\) rate almost exactly: the mean subspace error decreases from \(0.0141\) at \(m=50\) to \(0.0034\) at \(m=800\), with log--log slope \(-0.506\) and \(R^2=0.999\). The anisotropic-noise curve is higher for small and moderate \(m\); thus, in the linear-kernel
setting, anisotropy increases the finite-sample fluctuation constant
instead of creating a persistent population eigenspace bias.  In the long run, the decay of the rate remains of order
\(m^{-1/2}\), with constants depending on the spectral size and effective rank
of the noise covariance (Theorem~\ref{thm:main-pretrain} and Appendix~\ref{app:weighting}).  Consistently, the mean error decreases from
\(0.0318\) to \(0.0037\) and the
anisotropic curve approaches the isotropic curve as \(m\) grows.

Experiment~2 fixes heterogeneous feature dimensions \( (m_1,m_2,m_3,m_4)=(800,400,200,100) \) under the anisotropic noise setting and compares uniform weights, inverse-variance weights, and a simulation oracle.
The inverse-variance weighting rule is a simulation-level implementation of the
theory-guided inverse-variance weighting $w^{\rm IV}$ in~\eqref{eq:theory-iv-weights}. It assigns
\[
    w_s
    \propto
    \frac{\widehat\gamma_s}{\widehat v_s^2+10^{-9}}.
\]
We use the variance proxy
\[
    \widehat v_s^2
    =
  \widehat\sigma_{\infty,s}^2
        {\frac{\widehat B_{s,N}}{m_s}}
+\frac{\widehat\sigma_{\infty,s}^4}{{m_s}},
\]
where every quantity is computed from the observed matrix \(E_s\) only.
Writing \(R_s\) for the residual of \(E_s\) after its rank-\(d\) SVD,
\[
    \widehat\sigma_{\infty,s}^2=\frac{m_s\|R_s\|_{\op}^2}{(\sqrt{m_s}+\sqrt{N-d})^2},
    \qquad
    \widehat B_{s,N}=N\max_{j,j'}|\widehat K^{\rm sig}_s(j,j')|,
\]
where \(\widehat K^{\rm sig}_s\) is the top-\(d\) part of \(\frac1N E_sE_s^\top\)
with its eigenvalues shifted down by \(\|R_s\|_F^2/(N(N-d))\). The eigengap proxy
\(\widehat\gamma_s\) is the \(d\)-th eigengap of the observed kernel
\(
\frac{1}{N}E_sE_s^\top\).

For reference, we also report a simulation oracle, a grid-search diagnostic over a finite candidate set
of simplex weights (this is not guaranteed to be the exact optimizer over the
continuous simplex).  We use a dense candidate grid only to provide a
near-oracle benchmark for comparing weighting rules; it is not a feasible
procedure in real data because it uses the known \(Z_{\rm sig}\).

\subsection{End-to-end supervised risk}
Experiment~3 uses pretrained embeddings and learns bounded
relation-wise ReLU networks on a nonlinear triple-score teacher with
relation-specific bilinear interactions and Gaussian label noise
\(N(0,0.3^2)\).  We vary the labeled sample size \(n\), pretraining sample size \(m\), and the number of ReLU parameters per relation head \(W\); the results are shown in
Figure~\ref{fig:exp3}. The oracle baseline uses the true
latent entity coordinates from the simulator while keeping the same supervised training procedure.

The pretraining model in Experiment~3 uses the same linear-kernel multi-view
factor construction as in Experiments~1--2, with \(N=60\) entities,
latent rank \(d=4\), and \(S=4\) views.  The view-specific spectra
\(\Lambda_s\) are the same as above.  In the reported simulations,
the pretraining samples are generated under the isotropic-noise regime with \(\sigma=2.5\), using
uniform view weights. For each pretraining run, the embedding is obtained by forming the aggregated
linear kernel from the noisy entity matrices and taking its leading
\(d\)-dimensional eigenspace; all \(N=60\) rows, rescaled by \(\sqrt N\), are then
used as frozen entity representations for the downstream triple-score task.

We use $K=3$ relation types.
The true triple score is generated by a smooth one-hidden-layer tanh teacher with relation-specific bilinear terms, standardized to unit variance:
\[
    f^\star(h,r,t)
    =
    u^\top \tanh\!\left(
    A_{\rm teach}
    \begin{bmatrix}
    z_h\\ z_t\\ e_r
    \end{bmatrix}
    +b
    \right)
    +
    \langle \eta_r, z_h\odot z_t\rangle
    +c_r,
\]
where \(e_r\) is the one-hot encoding of relation \(r\), and
\(\eta_r\in\R^d\) and \(c_r\in\R\) are relation-specific
parameters.
Training labels are noisy observations,
\[
    Y_i=f^\star(h_i,r_i,t_i)+\varepsilon_i,
    \qquad
    \varepsilon_i\sim N(0,0.3^2).
\]
Test MSE is evaluated against the noiseless teacher score $f^\star(h,r,t)$ on all held-out triples, i.e., the triples among the $N^2K$ that were not drawn into the labeled sample.
Thus the reported risk measures estimation, approximation, and representation errors relative to the underlying true score function.

We compare two representations.
The oracle representation uses the true latent entity coordinates $Z_{\rm oracle}=Z_{\rm sig}$ and trains the same relation-wise ReLU heads as all other methods, thereby removing pretraining error and isolating the supervised approximation and estimation components.
The multi-view KPCA representation uses frozen coordinates estimated from the multi-view pretraining samples.
The gap between multi-view KPCA and the oracle measures the downstream cost of imperfect representation recovery.

Each relation head is a one-hidden-layer ReLU network; a head with \(q\) hidden units has \(W=q(2d+2)+1\) parameters.
The labeled-sample sweep fixes $m=40$ and $W=81$ and varies
\[
    n\in\{200,400,800,1600,3200,6400\}.
\]
The pretraining-sample-size sweep fixes $n=3000$ and $W=161$ and varies
\[
    m\in\{5,10,20,40,80,160,400,800\}.
\]
The parameter sweep fixes $n=900$ and $m=40$ and varies
\[
    W\in\{11,21,41,81,121,161,201,241\}.
\]
For each configuration we average over $20$ random seeds; each seed redraws the pretraining data, the labeled triples, the label noise, and the head initialization, and error bars show $\pm1$ standard deviation.

As the number of labeled triples increases,
test MSE decreases consistently with the supervised-estimation term
(Figure~\ref{fig:sim-main-exp3-n-risk}).
Multi-view KPCA remains at or above the oracle, with the gap representing the downstream
cost of imperfect coordinate recovery; this gap shrinks as the pretraining sample grows
(Figure~\ref{fig:sim-main-exp3-m-risk}).
The parameter sweep reveals a bias--variance trade-off in the relation heads: larger networks improve
expressiveness but raise estimation variance when labeled data are limited
(Figure~\ref{fig:sim-main-exp3-w-risk}).
All three trends match the risk decomposition in Theorem~\ref{thm:main-end2end}, where
supervised sample complexity, approximation capacity, and pretraining error jointly determine
the final risk.

\section{Real-Data Experimental Details}
\label{app:realdata}

This section gives the full real-data protocol for WordNet and RENKI-KG:
evaluation goals, dataset construction, side-information views and multi-view
initialization, the encoder-induced kernels, split rules, negative sampling, relation-balanced objectives
and metrics, and implementation details, followed by additional results.

\subsection{Evaluation goals}
\label{subsec:app-realdata-goals}

The real-data experiments test the empirical implications of the two-stage
framework. First, if entity-side information is aligned with the latent KG
geometry, then pretrained entity coordinates should be much stronger than
random coordinates for downstream relation-wise prediction. Second, when
multiple side-information views have heterogeneous quality, multi-view
aggregation can improve over a single view, and validation-selected weights
can outperform uniform weighting. We use two graphs with textual side
information: WordNet, a lexical-semantic graph whose synset definitions
provide clean text, and RENKI-KG, a biomedical graph with heterogeneous entity
types and type-specific relations, where encoders pretrained on different
corpora may carry complementary information.

\subsection{Datasets and preprocessing}
\label{subsec:app-realdata-datasets}

\paragraph{WordNet.}
We preprocess WordNet and evaluate it as a synset-level knowledge graph.
Starting from the WordNet pointer-relation data, we treat each retained synset
as an entity and each observed pointer relation between synsets as a positive
KG triple. The preprocessing consists of: (i) canonicalizing synset
identifiers so that all triples refer to synset-level nodes; (ii) retaining
semantic and lexical pointer relations as relation types; (iii) removing
triples whose head or tail synset lacks the textual information needed for
side-information encoding; and (iv) constructing a text description for each
retained synset from its synset type, representative lemma, and
gloss/definition. Each synset carries one of five part-of-speech types (noun,
verb, adjective, adjective satellite, adverb), which we use for type-matched
negative sampling. After preprocessing, the graph contains $116{,}566$
entities, $33$ relation types, and $373{,}330$ observed positive triples.
Table~\ref{app:tablew} gives the split statistics and relation-family summary.

\begin{table}[ht]
\centering
\caption{WordNet statistics after preprocessing and a $60\%/20\%/20\%$
random edge split. Nodes are WordNet synsets and edges are semantic or lexical
relations derived from WordNet pointers. The split sizes are identical across
seeds; the assignment of triples is redrawn for each seed.}
\label{app:tablew}
\begin{tabular}{lrp{7.8cm}}
\toprule
Quantity & Count & Description \\
\midrule
Entities & $116{,}566$ & Total number of retained synset nodes \\
Entity types & $5$ & Noun, verb, adjective, adjective-satellite, and adverb synsets \\
Relation types & $33$ & Number of retained base relation types \\
Observed positive triples & $373{,}330$ & Total number of synset-level positive triples before splitting \\
Train positive triples & $223{,}998$ & Positive triples assigned to training \\
Validation positive triples & $74{,}666$ & Positive triples assigned to validation \\
Test positive triples & $74{,}666$ & Positive triples assigned to testing \\
Node text & -- & Synset type, representative lemma, and gloss/definition \\
Major relation families & -- & Hypernymy/hyponymy, meronymy/holonymy, antonymy, derivational relations, domain relations, and other semantic or lexical WordNet pointer types \\
Most frequent relations & -- & \texttt{hypernym}, \texttt{hyponym}, \texttt{lexical\_derivationally\_related\_form}, \texttt{similar\_to}, \texttt{member\_holonym}, \texttt{member\_meronym}, \texttt{part\_holonym}, \texttt{part\_meronym}, \texttt{instance\_hypernym}, \texttt{instance\_hyponym} \\
\bottomrule
\end{tabular}
\end{table}

\paragraph{RENKI-KG.}
RENKI-KG is the medical KG constructed by \citet{liu2024representation}. It
links clinical concepts, diagnoses, and medications drawn from multiple
curated medical ontologies and drug databases, and it covers both narrative
and codified concepts used in electronic health records. Entities are UMLS
clinical concepts, phecodes, and RxNorm medications, and each of the $9$
relation types connects a fixed pair of entity types (e.g., concept--phecode,
medication--phecode, phecode--phecode). The side information of each entity
is its name string (e.g., ``phenobarbital'', ``foreign body injury'').
Table~\ref{app:table-renki} summarizes the graph; we refer to
\citet{liu2024representation} for the full construction.

\begin{table}[ht]
\centering
\caption{RENKI-KG statistics, following \citet{liu2024representation}.}
\label{app:table-renki}
\small
\begin{tabular}{lr}
\toprule
Quantity & Count \\
\midrule
Entities & $118{,}554$ \\
\quad Clinical concepts (UMLS CUIs) & $113{,}787$ \\
\quad Diagnosis codes (phecodes) & $1{,}846$ \\
\quad Medications (RxNorm) & $2{,}921$ \\
Relation types & $9$ \\
Observed positive triples & $698{,}368$ \\
\bottomrule
\end{tabular}
\end{table}

\subsection{Side-information views and multi-view initialization}
\label{subsec:app-realdata-views}

\paragraph{Encoders.}
On WordNet, each synset description is encoded by four pretrained language
models without task-specific fine-tuning: BERT-base (uncased), RoBERTa-base,
ELECTRA-base (discriminator), and DeBERTa-v3-base. Each view is the
mean-pooled last-layer hidden state (inputs truncated to $32$ tokens), with
rows $\ell_2$-normalized. On RENKI-KG, each entity name is encoded by seven
BERT-family encoders pretrained or fine-tuned on different corpora: BERT,
RoBERTa, NER-BERT, MentalBERT, SciBERT, BioBERT, and ClinicalBERT. Each view
is the last-layer \texttt{[CLS]} hidden state, with rows $\ell_2$-normalized.
Each encoder induces one view as in Appendix~\ref{app:encoder-implementation}.

\paragraph{Multi-view KPCA.}
Let $\widetilde E_s\in\R^{N\times m_s'}$ be the column-centered encoder
feature matrix of view $s$, with rows $(e_{s,j}-\bar e_s)^\top$ as in
Appendix~\ref{app:encoder-implementation}. Given weights $w\in\Delta_S$, the
shared coordinates are the top-$d$ eigenvectors of
$\sum_{s=1}^S (w_s/m_s')\widetilde E_s\widetilde E_s^\top$; we use $d=128$ on both graphs. A single-view
representation is the same construction with one view, i.e., its top-$d$
principal components. On RENKI-KG, the orthonormal coordinates are rescaled to
the mean row norm of the raw BERT embedding, the input scale of the RENKI head
network.

\paragraph{View weights.}
On WordNet, the weights are selected on validation AUROC from the $35$ points
of the simplex grid with step $0.25$, averaging two seeds ($101$ and $103$)
that are disjoint from the evaluation seeds. The selected weights are
$(1:2:1:0)$ over (BERT, RoBERTa, ELECTRA, DeBERTa-v3). On RENKI-KG, the seven views are
aggregated with uniform weights, without weight search.

\subsection{Encoder-induced kernels and coordinates}
\label{app:encoder-implementation}
Each language model is regarded as a kernel and introduces a view-specific feature map, denoted as
\[
    \phi_s:\mathcal U_s\to\mathbb R^{m_s'}.
\]
Let
\(u_{s,j}\in\mathcal U_s\) be the textual description of entity \(j\).
The encoder induces the entity kernel
\[
    k_s(u_{s,j},u_{s,j'})
    :=
    \left\langle
        \phi_s(u_{s,j}),
        \phi_s(u_{s,j'})
    \right\rangle_{\mathbb R^{m_s'}}.
\]
Writing
\(
    e_{s,j}:=\phi_s(u_{s,j})
\),
the corresponding kernel matrix is the centered Gram matrix of the explicit
encoder features (centering, which Algorithm~\ref{alg:multi-view-kpca} omits,
removes the common mean direction of the encoder outputs):
\[
    \widehat K_s(j,j')
    =
    \frac1N
    \langle e_{s,j}-\bar e_s,e_{s,j'}-\bar e_s\rangle,
    \qquad
    \bar e_s:=\frac1N\sum_{j=1}^N e_{s,j}.
\]

Accordingly, the encoder output is first interpreted as a feature map, and the
kernel matrix is reconstructed from pairwise similarities between the mapped
entity features.  This places the language-model views directly within the
multi-view entity-kernel framework developed above.

We aggregate the induced kernels as
\begin{equation}
\label{eq:main-realdata-linear-pca}
    \widehat K_{\rm mv}(w)
    :=
    \sum_{s=1}^S w_s\widehat K_s,
    \qquad
    w_s\ge0,
    \qquad
    \sum_{s=1}^S w_s=1.
\end{equation}

We use the leading \(d=128\) eigenvectors of the aggregated entity-kernel
matrix as the pretrained entity coordinates:
\begin{equation}
\label{eq:main-realdata-top-eig}
    \widehat K_{\rm mv}(w)\widehat\psi_\ell
    =
    \widehat\lambda_\ell\widehat\psi_\ell,
    \qquad
    \widehat\lambda_1
    \ge
    \widehat\lambda_2
    \ge\cdots,
    \qquad
    \widehat Z
    =
    [\widehat\psi_1,\ldots,\widehat\psi_{128}]
    \in\mathbb R^{N\times128}.
\end{equation}

\subsection{Evaluation protocol}
\label{subsec:app-realdata-splits}

\paragraph{Edge splits.}
On WordNet, the observed positive triples are randomly permuted and split
$60\%/20\%/20\%$ into train, validation, and test subsets. On RENKI-KG, each
triple is independently assigned to train, validation, or test with
probabilities $0.7/0.1/0.2$. In both cases the split is redrawn for each seed,
and all reported results are overall test performance under these splits.

\paragraph{Negative sampling.}
Both graphs contain observed positive triples only. We construct a balanced
binary link-prediction task by pairing each positive triple with one
corrupted negative triple. For a positive triple $(h,r,t)$, the negative
triple replaces the head or the tail (chosen with equal probability) while
keeping the relation fixed. The replacement entity is sampled uniformly from
the entity type of the corrupted endpoint: the synset type on WordNet, and the
type required by relation $r$ on RENKI-KG. Type-matched corruption prevents
the task from being solved through trivial type mismatch. Training negatives
are resampled for every mini-batch and are not filtered against observed
triples.

\paragraph{Relation-balanced objective and metrics.}
Both graphs have skewed relation frequencies, so a uniformly averaged loss can
be dominated by frequent relation types. We therefore use relation-balanced
weighting during training and evaluation. Let $\mathcal D_r$ denote the
examples with relation $r$ in the relevant split, and let
$n_r=|\mathcal D_r|$. We assign relation weights
\begin{equation}
\label{eq:app-relation-balanced-weight}
    \omega_r
    =
    \frac{n_r^{-1}}{\sum_{r'} n_{r'}^{-1}},
\end{equation}
and optimize the weighted empirical objective
\begin{equation}
\label{eq:app-relation-balanced-loss}
    \widehat L(\theta)
    =
    \sum_{(x_i,y_i)\in\mathcal D_{\rm train}}
    \omega_{r_i}\,\ell\{f_\theta(x_i),y_i\},
\end{equation}
where $x_i=(h_i,r_i,t_i)$ and $\ell$ is the squared error between the
sigmoid score and the binary label. The same objective is used for
\modelname{} and the KGE baselines. Relation weights computed on the
validation and test splits are used to aggregate the corresponding metrics.
We report relation-balanced AUROC, accuracy, and F1, with a positive
prediction defined as a sigmoid score of at least $0.5$. Results are averaged
over $10$ random seeds and shown as mean $\pm$ standard deviation.

\paragraph{Validation usage.}
The validation split is used only for model selection and is never included
in training or final test evaluation. For each model, the reported checkpoint
is the epoch with the best validation AUC: relation-balanced AUROC on
WordNet, and, following \citet{liu2024representation}, the average of
per-relation AUCs weighted by the number of validation triples on RENKI-KG.
On WordNet, the validation set is also used to select the view weights
(Appendix~\ref{subsec:app-realdata-views}). Once the checkpoint and view
weights are selected, all reported results are evaluated on the held-out test
split.

\paragraph{Models.}
In all text-derived \modelname{} variants, entity coordinates are frozen and
only the relation-specific heads are trained. On WordNet, each relation head is
an MLP $2d\to256\to128\to1$ with ReLU activations and dropout $0.1$ applied to
the concatenated head and tail coordinates; the random variant trains its
randomly initialized entity embeddings jointly with the head. On RENKI-KG, we use the RENKI
head of \citet{liu2024representation}, an MLP $2d\to256\to1$ with ReLU
activation per relation; the random variant trains its randomly initialized
entity embeddings jointly with the heads. The KGE baselines (TransE, DistMult,
ComplEx, RotatE) use dimension $128$ and are trained end to end with the same
splits, negatives, objective, and checkpoint selection. Optimization and
hardware details are given in Table~\ref{tab:downstream-kg-optimization-details}.

\begin{table}[t]
\centering
\small
\caption{Experimental details for WordNet and RENKI-KG. KGE baselines use
model-specific optimizers (Adam for TransE, DistMult, and RotatE; Adagrad for
ComplEx) and otherwise share the settings above.}
\label{tab:downstream-kg-optimization-details}
\begin{tabular}{p{0.30\linewidth}p{0.30\linewidth}p{0.30\linewidth}}
\toprule
\textbf{Detail} & \textbf{WordNet} & \textbf{RENKI-KG} \\
\midrule
Optimizer (\modelname{}) & AdamW & AdamW \\
Learning rate & $2\times 10^{-3}$ & $2\times 10^{-3}$ \\
Weight decay & $0$ & $0$ \\
Downstream KG batch size & $1024$ & $1024$ \\
Epochs & $50$ & $100$ \\
Coordinate dimension $d$ & $128$ & $128$ \\
Relation head & $2d\to256\to128\to1$, dropout $0.1$ & $2d\to256\to1$ \\
Encoder pooling & Mean, max.\ $32$ tokens & \texttt{[CLS]}, no truncation \\
Encoder embedding batch size & $32$ & $1$ \\
Data-loader workers & $2$ & -- (data on GPU) \\
Train/validation/test split & Random edge, $60/20/20$ & Random edge, $70/10/20$ \\
Number of random seeds & $10$ & $10$ \\
GPU resources & \multicolumn{2}{p{0.60\linewidth}}{NVIDIA GeForce RTX 5080 (16{,}303 MiB) and RTX 5070 (12{,}227 MiB)} \\
Driver/CUDA versions & \multicolumn{2}{p{0.60\linewidth}}{NVIDIA driver 615.71.09; CUDA 13.4} \\
Approximate execution time & About $10$ minutes per seed (four concurrent runs on two GPUs) & About $2$ minutes per seed (one run per GPU) \\
\bottomrule
\end{tabular}
\end{table}

\subsection{Full results on WordNet and RENKI-KG}
\label{subsec:app-realdata-results}

Tables~\ref{tab:app-wordnet-full} and~\ref{tab:app-renki-full} report all
methods on WordNet and RENKI-KG, including each single view, uniform
multi-view KPCA, and, on WordNet, multi-view KPCA with validation-selected
weights. On both graphs the single views differ markedly in quality: test
AUROC ranges from $0.8009$ (DeBERTa-v3) to $0.9441$ (BERT) on WordNet and from
$0.9172$ (NER-BERT) to $0.9607$ (SciBERT) on RENKI-KG. Except for DeBERTa-v3 on
WordNet, every text-derived view outperforms random coordinates. On WordNet,
the other three views also outperform every graph-only baseline on all three
metrics. On RENKI-KG, every view outperforms every graph-only baseline in
accuracy and F1; in AUROC, SciBERT, BioBERT, and ClinicalBERT exceed every
baseline, whereas BERT, RoBERTa, NER-BERT, and MentalBERT fall below RotatE. On WordNet, uniform weighting
($0.9388$ AUROC) falls slightly below the two strongest single views because
it assigns weight to the weak DeBERTa-v3 view. The validation-selected weights
$(1:2:1:0)$, which exclude that view, give the best result on all three
metrics. On RENKI-KG, where the strongest views are close in quality, uniform
multi-view KPCA alone is best on all three metrics.

\begin{table}[ht]
\centering
\caption{Full WordNet results (random edge split, relation-balanced test
metrics). Text-derived \modelname{} variants use fixed entity coordinates;
the random variant trains its entity embeddings. Mean $\pm$ standard
deviation over $10$ seeds. Best results are in bold.}
\label{tab:app-wordnet-full}
\small
\begin{tabular}{llccc}
\toprule
& Method & AUROC & Accuracy & F1 \\
\midrule
\multirow{4}{*}{KGE}
& TransE   & $0.8169 \pm 0.0221$ & $0.7616 \pm 0.0086$ & $0.6879 \pm 0.0147$ \\
& DistMult & $0.7274 \pm 0.0134$ & $0.6275 \pm 0.0136$ & $0.6742 \pm 0.0128$ \\
& ComplEx  & $0.6985 \pm 0.0142$ & $0.6005 \pm 0.0161$ & $0.6709 \pm 0.0099$ \\
& RotatE   & $0.5100 \pm 0.0155$ & $0.5622 \pm 0.0152$ & $0.6072 \pm 0.0147$ \\
\midrule
\multirow{7}{*}{\modelname{}}
& Random       & $0.8460 \pm 0.0065$ & $0.7517 \pm 0.0085$ & $0.6900 \pm 0.0178$  \\
& BERT        & $0.9441 \pm 0.0064$ & $0.8553 \pm 0.0114$ & $0.8451 \pm 0.0137$ \\
& RoBERTa     & $0.9421 \pm 0.0054$ & $0.8538 \pm 0.0112$ & $0.8423 \pm 0.0137$ \\
& ELECTRA     & $0.9195 \pm 0.0067$ & $0.8252 \pm 0.0121$ & $0.8126 \pm 0.0136$ \\
& DeBERTa-v3  & $0.8009 \pm 0.0074$ & $0.7119 \pm 0.0075$ & $0.6859 \pm 0.0134$ \\
& Multi-view (uniform)  & $0.9388 \pm 0.0057$ & $0.8455 \pm 0.0103$ & $0.8336 \pm 0.0135$ \\
& Multi-view (selected) & $\mathbf{0.9469} \pm 0.0071$ & $\mathbf{0.8597} \pm 0.0127$ & $\mathbf{0.8496} \pm 0.0144$ \\
\bottomrule
\end{tabular}
\end{table}

\begin{table}[ht]
\centering
\caption{Full RENKI-KG results (random edge split, relation-balanced test
metrics). Text-derived \modelname{} variants use fixed entity coordinates;
the random variant trains its entity embeddings. Mean $\pm$ standard
deviation over $10$ seeds. Best results are in bold.}
\label{tab:app-renki-full}
\small
\begin{tabular}{llccc}
\toprule
& Method & AUROC & Accuracy & F1 \\
\midrule
\multirow{4}{*}{KGE}
& TransE   & $0.9352 \pm 0.0016$ & $0.8333 \pm 0.0054$ & $0.8102 \pm 0.0077$ \\
& DistMult & $0.8986 \pm 0.0012$ & $0.7800 \pm 0.0026$ & $0.8087 \pm 0.0018$ \\
& ComplEx  & $0.7699 \pm 0.0037$ & $0.6944 \pm 0.0026$ & $0.7346 \pm 0.0023$ \\
& RotatE   & $0.9382 \pm 0.0018$ & $0.7993 \pm 0.0032$ & $0.7568 \pm 0.0050$ \\
\midrule
\multirow{9}{*}{\modelname{}}
& Random       & $0.8105 \pm 0.0035$ & $0.6314 \pm 0.0030$ & $0.4397 \pm 0.0071$ \\
& BERT         & $0.9326 \pm 0.0025$ & $0.8492 \pm 0.0052$ & $0.8345 \pm 0.0080$ \\
& RoBERTa      & $0.9364 \pm 0.0020$ & $0.8541 \pm 0.0045$ & $0.8395 \pm 0.0068$ \\
& NER-BERT     & $0.9172 \pm 0.0029$ & $0.8364 \pm 0.0039$ & $0.8186 \pm 0.0056$ \\
& MentalBERT   & $0.9225 \pm 0.0026$ & $0.8370 \pm 0.0035$ & $0.8185 \pm 0.0053$ \\
& SciBERT      & $0.9607 \pm 0.0016$ & $0.8871 \pm 0.0026$ & $0.8794 \pm 0.0033$ \\
& BioBERT      & $0.9556 \pm 0.0017$ & $0.8775 \pm 0.0027$ & $0.8681 \pm 0.0039$ \\
& ClinicalBERT & $0.9513 \pm 0.0019$ & $0.8730 \pm 0.0048$ & $0.8631 \pm 0.0067$ \\
& Multi-view (uniform) & $\mathbf{0.9624} \pm 0.0007$ & $\mathbf{0.8877} \pm 0.0046$ & $\mathbf{0.8804} \pm 0.0062$ \\
\bottomrule
\end{tabular}
\end{table}

\section{Disclosure of AI Assistance}
\label{sec:ai-disclosure}
We used AI tools to assist with language editing, formatting, and code
cleanup, and to check technical statements. These tools were not used to develop the theoretical results or
their proofs, or to generate the original scientific content. The authors
reviewed all AI-assisted edits and take full responsibility for the content
of the paper.


\clearpage

\begin{thebibliography}{33}
\providecommand{\natexlab}[1]{#1}
\providecommand{\url}[1]{\texttt{#1}}
\expandafter\ifx\csname urlstyle\endcsname\relax
  \providecommand{\doi}[1]{doi: #1}\else
  \providecommand{\doi}{doi: \begingroup \urlstyle{rm}\Url}\fi

\bibitem[Anthony and Bartlett(1999)]{anthony1999neural}
Martin Anthony and Peter~L. Bartlett.
\newblock \emph{Neural Network Learning: Theoretical Foundations}.
\newblock Cambridge University Press, 1999.

\bibitem[Bonner et~al.(2022)Bonner, Barrett, Ye, Swiers, Engkvist, Bender, Hoyt, and Hamilton]{bonner2022review}
Stephen Bonner, Ian~P Barrett, Cheng Ye, Rowan Swiers, Ola Engkvist, Andreas Bender, Charles~Tapley Hoyt, and William~L Hamilton.
\newblock A review of biomedical datasets relating to drug discovery: a knowledge graph perspective.
\newblock \emph{Briefings in Bioinformatics}, 23\penalty0 (6):\penalty0 bbac404, 2022.

\bibitem[Bordes et~al.(2013)Bordes, Usunier, Garcia-Duran, Weston, and Yakhnenko]{bordes2013translating}
Antoine Bordes, Nicolas Usunier, Alberto Garcia-Duran, Jason Weston, and Oksana Yakhnenko.
\newblock Translating embeddings for modeling multi-relational data.
\newblock \emph{Advances in Neural Information Processing Systems}, 26, 2013.

\bibitem[Cho and Saul(2009)]{cho2009kernel}
Youngmin Cho and Lawrence Saul.
\newblock Kernel methods for deep learning.
\newblock \emph{Advances in Neural Information Processing Systems}, 22, 2009.

\bibitem[Daza et~al.(2021)Daza, Cochez, and Groth]{daza2021inductive}
Daniel Daza, Michael Cochez, and Paul Groth.
\newblock Inductive entity representations from text via link prediction.
\newblock In \emph{Proceedings of The Web Conference}, pages 798--808, 2021.

\bibitem[El~Karoui(2010)]{el2010spectrum}
Noureddine El~Karoui.
\newblock The spectrum of kernel random matrices.
\newblock \emph{The Annals of Statistics}, 38\penalty0 (1):\penalty0 1--50, 2010.

\bibitem[Fan et~al.(2019)Fan, Wang, Wang, and Zhu]{fan2019distributed}
Jianqing Fan, Dong Wang, Kaizheng Wang, and Ziwei Zhu.
\newblock Distributed estimation of principal eigenspaces.
\newblock \emph{The Annals of Statistics}, 47\penalty0 (6):\penalty0 3009, 2019.

\bibitem[Firey(1960)]{firey1960remainder}
William~J Firey.
\newblock Remainder formulae in Taylor's theorem.
\newblock \emph{The American Mathematical Monthly}, 67\penalty0 (9):\penalty0 903--905, 1960.

\bibitem[Ge et~al.(2024)Ge, Tang, Fan, and Jin]{ge2023provable}
Jiawei Ge, Shange Tang, Jianqing Fan, and Chi Jin.
\newblock On the provable advantage of unsupervised pretraining.
\newblock In \emph{The Twelfth International Conference on Learning Representations}, 2024.

\bibitem[HaoChen et~al.(2021)HaoChen, Wei, Gaidon, and Ma]{haochen2021provable}
Jeff~Z. HaoChen, Colin Wei, Adrien Gaidon, and Tengyu Ma.
\newblock Provable guarantees for self-supervised deep learning with spectral contrastive loss.
\newblock In \emph{Advances in Neural Information Processing Systems}, volume~34, pages 5000--5011, 2021.

\bibitem[Hoff et~al.(2002)Hoff, Raftery, and Handcock]{hoff2002latent}
Peter~D. Hoff, Adrian~E. Raftery, and Mark~S. Handcock.
\newblock Latent space approaches to social network analysis.
\newblock \emph{Journal of the American Statistical Association}, 97\penalty0 (460):\penalty0 1090--1098, 2002.

\bibitem[Huang et~al.(2019)Huang, Zhang, Li, and Li]{huang2019knowledge}
Xiao Huang, Jingyuan Zhang, Dingcheng Li, and Ping Li.
\newblock Knowledge graph embedding based question answering.
\newblock In \emph{Proceedings of the Twelfth {ACM} International Conference on Web Search and Data Mining}, pages 105--113, 2019.

\bibitem[Ji et~al.(2022)Ji, Pan, Cambria, Marttinen, and Yu]{ji2021survey}
Shaoxiong Ji, Shirui Pan, Erik Cambria, Pekka Marttinen, and Philip~S Yu.
\newblock A survey on knowledge graphs: Representation, acquisition, and applications.
\newblock \emph{{IEEE} Transactions on Neural Networks and Learning Systems}, 33\penalty0 (2):\penalty0 494--514, 2022.

\bibitem[Lee et~al.(2021)Lee, Lei, Saunshi, and Zhuo]{lee2021predicting}
Jason~D. Lee, Qi~Lei, Nikunj Saunshi, and Jiacheng Zhuo.
\newblock Predicting what you already know helps: Provable self-supervised learning.
\newblock In \emph{Advances in Neural Information Processing Systems}, volume~34, pages 309--323, 2021.

\bibitem[Liu et~al.(2024)Liu, Cai, and Li]{liu2024representation}
Suqi Liu, Tianxi Cai, and Xiaoou Li.
\newblock Representation-enhanced neural knowledge integration with application to large-scale medical ontology learning.
\newblock \emph{{arXiv} preprint {arXiv}:2410.07454}, 2024.

\bibitem[MacDonald et~al.(2022)MacDonald, Levina, and Zhu]{macdonald2022latent}
Peter~W. MacDonald, Elizaveta Levina, and Ji~Zhu.
\newblock Latent space models for multiplex networks with shared structure.
\newblock \emph{Biometrika}, 109\penalty0 (3):\penalty0 683--706, 2022.

\bibitem[Miller(1995)]{miller1995wordnet}
George~A Miller.
\newblock {WordNet}: a lexical database for English.
\newblock \emph{Communications of the {ACM}}, 38\penalty0 (11):\penalty0 39--41, 1995.

\bibitem[Nickel et~al.(2016)Nickel, Murphy, Tresp, and Gabrilovich]{nickel2016review}
Maximilian Nickel, Kevin Murphy, Volker Tresp, and Evgeniy Gabrilovich.
\newblock A review of relational machine learning for knowledge graphs.
\newblock \emph{Proceedings of the {IEEE}}, 104\penalty0 (1):\penalty0 11--33, 2016.

\bibitem[Nourdin et~al.(2010)Nourdin, Peccati, and Reinert]{nourdin2010invariance}
Ivan Nourdin, Giovanni Peccati, and Gesine Reinert.
\newblock Invariance principles for homogeneous sums: universality of Gaussian Wiener chaos.
\newblock \emph{The Annals of Probability}, 38\penalty0 (5):\penalty0 1947--1985, 2010.
\newblock \doi{10.1214/10-AOP531}.

\bibitem[Saunshi et~al.(2019)Saunshi, Plevrakis, Arora, Khodak, and Khandeparkar]{arora2019theoretical}
Nikunj Saunshi, Orestis Plevrakis, Sanjeev Arora, Mikhail Khodak, and Hrishikesh Khandeparkar.
\newblock A theoretical analysis of contrastive unsupervised representation learning.
\newblock In \emph{International Conference on Machine Learning}, pages 5628--5637. {PMLR}, 2019.

\bibitem[Sun et~al.(2019)Sun, Deng, Nie, and Tang]{sun2019rotate}
Zhiqing Sun, Zhi-Hong Deng, Jian-Yun Nie, and Jian Tang.
\newblock {RotatE}: Knowledge graph embedding by relational rotation in complex space.
\newblock In \emph{International Conference on Learning Representations}, 2019.

\bibitem[Trouillon et~al.(2016)Trouillon, Welbl, Riedel, Gaussier, and Bouchard]{trouillon2016complex}
Th{\'e}o Trouillon, Johannes Welbl, Sebastian Riedel, {\'E}ric Gaussier, and Guillaume Bouchard.
\newblock Complex embeddings for simple link prediction.
\newblock In \emph{International Conference on Machine Learning}, pages 2071--2080. {PMLR}, 2016.

\bibitem[Wang et~al.(2021{\natexlab{a}})Wang, Shen, Long, Zhou, Wang, and Chang]{wang2021star}
Bo~Wang, Tao Shen, Guodong Long, Tianyi Zhou, Ying Wang, and Yi~Chang.
\newblock Structure-augmented text representation learning for efficient knowledge graph completion.
\newblock In \emph{Proceedings of The Web Conference}, pages 1737--1748, 2021{\natexlab{a}}.

\bibitem[Wang et~al.(2022)Wang, Zhao, Wei, and Liu]{wang2022simkgc}
Liang Wang, Wei Zhao, Zhuoyu Wei, and Jingming Liu.
\newblock {SimKGC}: Simple contrastive knowledge graph completion with pre-trained language models.
\newblock In \emph{Proceedings of the 60th Annual Meeting of the Association for Computational Linguistics (Volume 1: Long Papers)}, pages 4281--4294, 2022.

\bibitem[Wang et~al.(2021{\natexlab{b}})Wang, Gao, Zhu, Zhang, Liu, Li, and Tang]{wang2021kepler}
Xiaozhi Wang, Tianyu Gao, Zhaocheng Zhu, Zhengyan Zhang, Zhiyuan Liu, Juanzi Li, and Jian Tang.
\newblock {KEPLER}: A unified model for knowledge embedding and pre-trained language representation.
\newblock \emph{Transactions of the Association for Computational Linguistics}, 9:\penalty0 176--194, 2021{\natexlab{b}}.

\bibitem[Wilson et~al.(2016)Wilson, Hu, Salakhutdinov, and Xing]{wilson2016deep}
Andrew~Gordon Wilson, Zhiting Hu, Ruslan Salakhutdinov, and Eric~P Xing.
\newblock Deep kernel learning.
\newblock In \emph{Artificial Intelligence and Statistics}, pages 370--378. PMLR, 2016.

\bibitem[Xie et~al.(2016)Xie, Liu, Jia, Luan, and Sun]{xie2016representation}
Ruobing Xie, Zhiyuan Liu, Jia Jia, Huanbo Luan, and Maosong Sun.
\newblock Representation learning of knowledge graphs with entity descriptions.
\newblock In \emph{Proceedings of the {AAAI} Conference on Artificial Intelligence}, volume~30, 2016.

\bibitem[Yan et~al.(2024)Yan, Chen, and Fan]{yan2024inference}
Yuling Yan, Yuxin Chen, and Jianqing Fan.
\newblock Inference for heteroskedastic {PCA} with missing data.
\newblock \emph{The Annals of Statistics}, 52\penalty0 (2):\penalty0 729--756, 2024.

\bibitem[Yang et~al.(2015)Yang, Yih, He, Gao, and Deng]{yang2014embedding}
Bishan Yang, Wen-tau Yih, Xiaodong He, Jianfeng Gao, and Li~Deng.
\newblock Embedding entities and relations for learning and inference in knowledge bases.
\newblock In \emph{International Conference on Learning Representations}, 2015.

\bibitem[Yang et~al.(2022)Yang, Wu, Yang, Lian, Guo, and Wang]{yang2022survey}
Yang Yang, Zhilei Wu, Yuexiang Yang, Shuangshuang Lian, Fengjie Guo, and Zhiwei Wang.
\newblock A survey of information extraction based on deep learning.
\newblock \emph{Applied Sciences}, 12\penalty0 (19):\penalty0 9691, 2022.

\bibitem[Yao et~al.(2019)Yao, Mao, and Luo]{yao2019kg}
Liang Yao, Chengsheng Mao, and Yuan Luo.
\newblock {KG-BERT}: {BERT} for knowledge graph completion.
\newblock \emph{{arXiv} preprint {arXiv}:1909.03193}, 2019.

\bibitem[Yarotsky(2017)]{yarotsky2017error}
Dmitry Yarotsky.
\newblock Error bounds for approximations with deep {ReLU} networks.
\newblock \emph{Neural Networks}, 94:\penalty0 103--114, 2017.

\bibitem[Yu et~al.(2015)Yu, Wang, and Samworth]{yu2015useful}
Yi~Yu, Tengyao Wang, and Richard~J Samworth.
\newblock A useful variant of the {Davis--Kahan} theorem for statisticians.
\newblock \emph{Biometrika}, 102\penalty0 (2):\penalty0 315--323, 2015.

\bibitem[Zhang et~al.(2016)Zhang, Yuan, Lian, Xie, and Ma]{zhang2016collaborative}
Fuzheng Zhang, Nicholas~Jing Yuan, Defu Lian, Xing Xie, and Wei-Ying Ma.
\newblock Collaborative knowledge base embedding for recommender systems.
\newblock In \emph{Proceedings of the 22nd {ACM} {SIGKDD} International Conference on Knowledge Discovery and Data Mining}, pages 353--362, 2016.

\end{thebibliography}
\end{document}